\documentclass[a4paper, 10pt]{article}
\date{}
\usepackage{amsthm,amsmath,amssymb,mathrsfs,mathtools,picture,graphicx,color,hyperref}
\hypersetup{colorlinks=true,linkcolor=red,citecolor=blue,filecolor=magenta,urlcolor=cyan}

\allowdisplaybreaks

\def\_email#1@#2\q_nil{\href{mailto:#1@#2}{{\emailfont #1\emailampersat #2}}}
\newcommand\emailampersat{{\color{cyan}\small@}}

\newtheorem{thm}{Theorem}[section]
\newtheorem{lma}[thm]{Lemma}
\newtheorem{coro}[thm]{Corollary}
\newtheorem{remark}{Remark}

\numberwithin{equation}{section}

\newcommand{\rpal}{r_{p,\alpha}^{\boldsymbol{\lambda}}}
\newcommand{\blambda}{\boldsymbol{\lambda}}
\newcommand{\bmu}{\boldsymbol{\mu}}
\title{\bf Near- and far-field expansions for stationary  solutions of Poisson--Nernst--Planck equations\footnote{Due to inelegant symbol, we provide an original manuscript in arXiv.}}
\author{Jhih-Hong Lyu \thanks{Department of Mathematics, National Taiwan University, Taipei 10617, Taiwan (\tt d06221001@ntu.edu.tw).}, Chiun-Chang Lee \thanks{Institute for Computational and Modeling Science, National Tsing Hua University, Hsinchu 30013, Taiwan ({\tt chlee@mail.nd.nthu.edu.tw}).}, Tai-Chia Lin \thanks{Department of Mathematics, National Taiwan University, Taipei 10617, Taiwan; National Center for Theoretical Sciences, Mathematics Division, Taipei 10617, Taiwan ({\tt tclin@math.ntu.edu.tw}).}}
\begin{document}
\maketitle
\begin{abstract} 
This work is concerned with the stationary Poisson--Nernst--Planck equation with a large parameter which describes a huge number of ions occupying an electrolytic region. Firstly, we focus on the model with a single specie of positive charges in one-dimensional bounded domains due to the assumption that these ions are transported in the same direction along a tubular-like mircodomain, as was introduced in \cite{Cartailler2017}. We show that the solution asymptotically blows up in a thin region attached to the boundary, and establish the refined ``near-field'' and ``far-field'' expansions for the solutions with respect to the parameter. Moreover, we obtain the boundary concentration phenomenon of the net charge density, which mathematically confirms the physical description that the non-neutral phenomenon occurs near the charged surface. In addition, we revisit a nonlocal Poisson--Boltzmann model for monovalent binary ions (cf. \cite{Lee2019,Lee2011}) and establish a novel comparison for these two models.
\vspace{5mm}
\vspace{1mm}
\end{abstract}

\section{Introduction}\label{sec1}
With the development of current nanotechnology in electrochemistry~\cite{HY2015,RZSMB2012},
analysis of solutions to mathematical models with regulated parameters plays a more and more important role in investigating the distribution of electrostatic potential in electrolyte-like solutions (cf. \cite{FKL2017,HN1995,Hcy2019-1,HEL2011,KJH2005,MCNK2003,ST2015,SMP2016,PE2007}). 
In the past few decades, a key ingredient for understanding of such microscopic phenomena is the Poisson--Boltzmann equation \cite{DH1923,S2012,WXLS2014},
but which suffered from deficiencies that are well known. Different types of nonlinear electrochemistry systems (see, e.g.,~\cite{Holcman2018,HR2015,HL2015-1,HL2015-2,Hcy2019-2,HEL2014,HKL2010,Lee2021,LE2014,LE2015,MOC1986,R2008,RLW2006,SGNE2008}) have already been modeled and the corresponding computational confirmation is expected to reach theoretical prediction.

Among such phenomena in electrochemistry, an unstable one relates to the non-electroneutrality characterized by the localization of excess charge distributions. For instance, a faradaic current drives ions from one electrode surface to another and results in an ionic distribution which is non-uniform near the electrode surface. We refer the reader to~\cite{FK1987,non-neual}. Other  instabilities including the electroconvective instability in channels have been investigated; see, e.g., \cite{Liu2020} and references therein.
Despite the intensive investigation yielding detailed information and the reasonable success, little is known about the non-electroneutral phenomenon near the charged surface.
In \cite{Cartailler2017,CSH2017}, employing the one dimensional Poisson--Nernst--Planck (PNP) system provides an underlying framework for such phenomena, where it is reasonably assumed that all ions transport in the same direction along a tubular-like mircodomain with finite-length so that the physical domain can be set as a one-dimensional interval $\Omega=(0,1)$.
At the beginning of this work we consider the model for single-ion species which is represented as
\begin{equation}\label{eq:1.01}\left\{\begin{aligned}&\quad\,\,\,\,\frac{\partial{p}}{\partial{t}}(t,r)=D_p\frac{\partial}{\partial{r}}\left(\frac{\partial{p}}{\partial{r}}(t,r)+\frac{z_p\mathrm{e}_0}{k_BT}p(t,r)\frac{\partial{u}}{\partial{r}}(t,r)\right),\\&-\frac{\partial^2{u}}{\partial{r}^2}(t,r)=\blambda(\rho_{\text{f}}(r)+z_p\mathrm{e}_0p(t,r)),\quad\quad\quad(t,r)\in(0,\infty)\times(0,1),\end{aligned}\right.
\end{equation}
with the no-flux boundary conditions
\begin{equation}\label{eq:1.02}
\left(\frac{\partial{p}}{\partial{r}}+\frac{z_p\mathrm{e}_0}{k_BT}p\frac{\partial{u}}{\partial{r}}\right)(t,0)=\left(\frac{\partial{p}}{\partial{r}}+\frac{z_p\mathrm{e}_0}{k_BT}p\frac{\partial{u}}{\partial{r}}\right)(t,1)=0,\,\,t>0,
\end{equation}
and initial conditions $p(0,r)=p_0(r)>0$ and $u(0,r)=u_0(r)$, for $r\in[0,1]$, which are compatible with the boundary conditions \eqref{eq:1.02}.
Notice that \eqref{eq:1.01}--\eqref{eq:1.02} satisfies shift in variance in $u$ and so we will set the boundary condition for $u$ later.

Physically, $u$ is the electrostatic potential, $p$ is the density of cations, $D_p$ is the diffusion coefficient, $z_p$ is the valency of cations, $\mathrm{e}_0$ is the elementary charge, $k_B$ is the Boltzmann constant, $T$ is the absolute temperature, and $\rho_{\text{f}}$ is the fixed permanent charge density in the physical domain.
Besides, $\blambda=d^2\mathrm{e}_0S/(\varepsilon_0U_T)$ where $\varepsilon_0$ is the dielectric constant of the electrolyte, $U_T$ is the thermal voltage, $d$ is the diameter of the domain $\Omega=\left(0,1\right)$, and $S$ is the appropriate concentration scale \cite{PW2017}.
Note that the no-flux boundary conditions~\eqref{eq:1.02} ensure the conservation of ions, i.e.,
\begin{equation}\label{eq:1.03}
\int_0^1p(t,r)\mathrm{d}r=\int_0^1p_0(r)\mathrm{d}r,\,\,\forall\,t>0.
\end{equation}

To see the non-electroneutral phenomenon near the charged surface, $\blambda$ is assumed a large positive parameter corresponding to the number of ions occupying an electrolytic region, as was studied in \cite{Cartailler2017}.
To basically understand the non-electroneutral phenomenon of $u$ and $p$ as $\blambda\gg1$, we neglect the effect of fixed permanent charges by setting $\rho_{\text{f}}=0$ and focus mainly on their steady-state solutions.
Under the standard dimensionless formulation, we may set 
\begin{equation}\label{eq:1.04}
z_p=1,\,\,\mathrm{e}_0/k_BT=1,\,\,D_p=1,
\end{equation}
and regard $\blambda$ as a positive parameter. 
Without loss of generality, we may assume $\int_{0}^1p_0(r)\mathrm{d}r=1$. Then by \eqref{eq:1.03} and \eqref{eq:1.04}, the steady-states of \eqref{eq:1.01}--\eqref{eq:1.02} verify $u(t,r)=u_{\blambda}(r)$ and the nonlocal relation $p(t,r)=\frac{e^{-u_{\blambda}(r)}}{\int_0^1e^{-u_{\blambda}(s)}\text{d}s}$ (see also, \cite[equation (1.1) with $p=1$]{JACarrillo1998}, \cite[Section 2.1]{Cartailler2017} and \cite{Lee2011} and references therein).
As a consequence, $u_{\blambda}$ satisfies
\begin{align}\label{eq:1.05}
-u_{\blambda}''(r)=\rho_{\blambda}(r):=\frac{\blambda e^{-u_{\blambda}(r)}}{\int_0^1 e^{-u_{\blambda}(s)}\mbox{d}s},~r\in(0,1).
\end{align}
Physically, $\rho_{\blambda}$ represents the net charge density of the ion species and the parameter $\int_0^1\rho_{\blambda}(r)\mbox{d}r$ denotes the total charges of ions. It should be mentioned that in \cite[Chapter 10.6.1]{HS2018}, \eqref{eq:1.05} is named the reverse Liouville--Brat\'{u}--Gelfand (LBG) equation which can be regarded as the nonlocal LBG equation~(see, e.g., \cite{JACOBSEN2002283} and references therein).

For \eqref{eq:1.05}, we consider the following condition at $r=0$: 
\begin{equation}\label{eq:1.06}
u_{\blambda}\left(0\right)=u_{\blambda}'\left(0\right)=0,
\end{equation}
as was presented in \cite{Cartailler2017}, where assumption $u_{\blambda}(0)=0$ is actually due to a fact that equation \eqref{eq:1.05} satisfies the shift invariance; that is, for any $c\in\mathbb{R}$, $u_{\blambda}+c$ also satisfies \eqref{eq:1.05}.
We emphasize that
such a setting immediately implies 
\begin{equation}\label{eq:1.07}
u_{\blambda}'(1)=-\blambda.
\end{equation}
This along with \eqref{eq:1.06} asserts that there exists an interior point $r^{\blambda}_{\mathrm{int}}\in(0,1)$ depending on $\blambda$ such that when $\blambda\to\infty$, the net charge density $\rho_{\blambda}(r^{\blambda}_{\mathrm{int}})=-u_{\blambda}''(r^{\blambda}_{\mathrm{int}})=u_{\blambda}'(0)-u_{\blambda}'(1)=\blambda$ asymptotically blows up.
As will be clarified later, such a phenomenon occurs only when $r^{\blambda}_{\mathrm{int}}$ is sufficiently close to the boundary point $r=1$ as $\blambda$ tends to infinity.
The boundary blow-up behavior theoretically represents the non-neutral phenomenon near the charged surface.
We further refer the reader to \cite{Lee2019,Lee2011,ccleerolf2018} for the different boundary blow-up phenomena related models with multiple species under various boundary conditions.

It should be stress that equation \eqref{eq:1.05}--\eqref{eq:1.06} has recently been studied by J. Cartailler et al \cite{Cartailler2017}, who used the phase-plane analysis to obtain the uniqueness of equation \eqref{eq:1.05}--\eqref{eq:1.06}.
A more general argument for the uniqueness can also be found in \cite{Lee2014}.
Moreover, the authors in \cite{Cartailler2017} showed that as $\blambda\to\infty$, the numerical solution to \eqref{eq:1.05}--\eqref{eq:1.07} develops boundary layers near $r=1$.
They further studied the leading order term of the asymptotic expansions of $u_{\blambda}(1)$ with respect to $\blambda\gg1$, and compared the result with their numerical simulation.
\textit{However, the more refined pointwise asymptotics for solutions, which is crucial for describing the sharp change of $u_{\blambda}$ near the boundary, remains unclear.}
The issue will be addressed in detail in this work. 

\subsection{Near-field and far-field expansions}
\label{subsec:1.1}

The main difficulty lies in the lack of available asymptotic analysis for such singularly perturbed nonlocal models, which we shall explain as follows.
Firstly, one observes from \eqref{eq:1.05} and \eqref{eq:1.06} that
\begin{align}
\label{eq:1.08}
u_{\blambda}(r),\,\,u_{\blambda}'(r),\,\,u_{\blambda}''(r)<0,~~r\in(0,1).
\end{align}
In particular, \eqref{eq:1.05}--\eqref{eq:1.07} and \eqref{eq:1.08} imply that $\int_0^1 e^{-u_{\blambda}(s)}\mbox{d}s\geq1$ and
\begin{align}\label{eq:1.09}
u_{\blambda}(1)=-\log\left(1+\frac{\blambda}{2}\int_0^1 e^{-u_{\blambda}(s)}\mbox{d}s\right)\to-\infty\,\,\mbox{as}\,\,\blambda\to\infty.
\end{align}
This illustrates the importance of the refined asymptotic expansions of $\int_0^1e^{-u_{\blambda}(s)}\mbox{d}s$ with respect to $\blambda\gg1$.  A perspective on obtaining such asymptotics is to consider the following nonlinear eigenvalue problem  with $\kappa>0$ (see also, \cite[Appendix]{Cartailler2017}):
\begin{equation}\label{eq:1.10}
\left\{
\begin{aligned}
&-U_{\kappa}''(r)=\kappa e^{-U_{\kappa}(r)},\quad r\in(0,1),
\\&
U_{\kappa}(0)=U_{\kappa}'(0)=0.\end{aligned}\right.
\end{equation}
When $\kappa\in(0,\pi^2/2)$, \eqref{eq:1.10} has a unique solution $U_{\kappa}(r)=2\log\cos(r\sqrt{\kappa/2})$  which is uniformly bounded on $[0,1]$, and $U_{\kappa}(1)\to-\infty$ as $\kappa$ approaches $\pi^2/2$ from the left. However, when $\kappa>\pi^2/2$, \eqref{eq:1.10} does not have solutions defined in the whole domain $(0,1)$.  As a consequence, by \eqref{eq:1.09} it is expected that a solution~$u_{\blambda}$ of \eqref{eq:1.05}--\eqref{eq:1.06} satisfies
\begin{align}\label{eq:1.11}
    \frac{\blambda}{\int_0^1e^{-u_{\blambda}(s)}\mathrm{d}s}<\frac{\pi^2}{2}\,\,\mathrm{for}\,\,\blambda>0,\,\,\mathrm{and}\,\,\frac{\blambda}{\int_0^1e^{-u_{\blambda}(s)}\mathrm{d}s}\stackrel{\blambda\to\infty}{-\!\!\!-\!\!\!-\!\!\!-\!\!\!\rightarrow}\frac{\pi^2}{2}.
\end{align} 
We will prove \eqref{eq:1.11} rigorously and establish a more refined asymptotic expansion of $\int_0^1e^{-u_{\blambda}(s)}\mathrm{d}s$ with respect to $\blambda\gg1$; see the detail in  \eqref{eq:2.01} and Remark~\ref{rk-20-1111}. Moreover, it yields the limiting equation for \eqref{eq:1.05}--\eqref{eq:1.06} formally as follows:
\begin{align}
\label{eq:1.12}
\begin{cases}
-U''(r)=\displaystyle\frac{\pi^2}{2}e^{-U(r)},~~r\in(0,1),\\
U(0)=U'(0)=0,
\end{cases}
\end{align}
and the unique solution
\begin{align}
\label{eq:1.13}
U(r):=U_{\frac{\pi^2}{2}}(r)=2\log\cos\left(\frac{\pi}{2}r\right),\,\,r\in[0,1),
\end{align}
represents the zeroth-order outer-solution to \eqref{eq:1.05}--\eqref{eq:1.06} with respect to $\blambda\gg1$; see also, Remark~\ref{remark1}. In passing we note that $U(r)\to-\infty$ and $U'(r)=-\pi\tan(\frac{\pi}{2}r)\to-\infty$ as $r\uparrow1$, which formally coincide with the boundary asymptotic blow-up behavior of $u_{\blambda}$ and $u_{\blambda}'$ obtained respectively in \eqref{eq:1.09} and \eqref{eq:1.07}. In summary, equation \eqref{eq:1.05}--\eqref{eq:1.06} is a nonlocal singularly perturbed model with small parameter $\frac{1}{\blambda}$ in front of its Laplace operator, and its limiting equation connects to a boundary blow-up problem \eqref{eq:1.12}.

What we want to point out is that the limiting model \eqref{eq:1.12} of \eqref{eq:1.05}--\eqref{eq:1.06} cannot be directly obtained from applying the standard method of matched asymptotic expansions since the nonlocal coefficient $\frac{\blambda}{\int_0^1 e^{-u_{\blambda}(s)}\text{d}s}$ is involved with the parameter $\blambda$ and its corresponding zeroth order outer solution $U(r)$ diverges near $r=1$.
To achieve the pointwise description of $u_{\blambda}$ with respect to $\blambda\gg1$, our first task is concerned with the asymptotic expansions for the solution $u_{\blambda}$, consisting of the near-field expansion and the far-field expansion.
In order for the reader to realize the difference from the method of matched asymptotic expansions, we present the argument as follows:
\begin{itemize}
\item The \textbf{near-field expansion} focuses mainly on the refined asymptotics of $u_{\blambda}(\rpal)$ and $u'_{\blambda}(\rpal)$, where $\alpha$ and $p$ are positive constants independent of $\blambda$, and
\begin{equation}\label{eq:1.14}
\rpal=1-\frac{p}{\blambda^\alpha}\in(0,1),\,\,\blambda\gg1,
\end{equation}
is sufficiently close to the boundary point $r=1$.
Let us emphasize again that as $\blambda\gg1$, $u_{\blambda}$ develops boundary layers and asymptotically blows up near the boundary $r=1$ so that the asymptotic behavior of $u_{\blambda}$ has dramatic changes in a thin neighborhood of $r=1$.
To better understand the structure of boundary layers, we are devoted to the pointwise asymptotics of $u_{\blambda}(\rpal)$ and $u_{\blambda}'(\rpal)$ with various $\alpha$ and $p$; see, for example, \eqref{eq:1.18} and \eqref{eq:1.19}.
Such a consideration essentially points out the difference between the analysis of \eqref{eq:1.05}--\eqref{eq:1.06} and the standard singularly perturbed equations.
\item the \textbf{far-field expansion} focuses on the refined asymptotics for $u_{\blambda}$ in $\mathcal{C}^1(K)$  as $\blambda\to\infty$, where $K$ independent of $\blambda$ is a compact subset of $[0,1)$.
We will frequently use the norm $\|\cdot\|_{\mathcal{C}^1(D)}$ defined by\begin{equation}\label{eq:1.15}\|f\|_{\mathcal{C}^1(D)}:=\sup_{D}\left(|f|+|f'|\right)\end{equation}
where $D$ is a bounded set of $\mathbb{R}$ and $f\in\mathcal{C}^1(D)$. When $D=K$ is a compact subset of $[0,1)$, \eqref{eq:1.15} is used for convenience to describe the asymptotic expansions of $u$ in $K$.
\end{itemize} 
We refer the reader to \cite{Han2013,Xing2011} for the more detailed physical background of these two terminologies.

We are now in a position to draw the asymptotic behavior of $u_{\blambda}$.
As will be presented in detail in Theorem~\ref{thm:2.1}, for any compact subset $K$ (independent of $\blambda$) of $[0,1)$, we establish the \textbf{far-field expansion} of $u_{\blambda}$ with the precise asymptotic expansions up to the order of $\frac{1}{\blambda^2}$:
\begin{align}\label{eq:1.16}
u_{\blambda}(r)=\,U(r)+\pi{r}\left(\frac{\sin(\pi r)}{\blambda}-\frac{\pi{r}+2\sin(\pi r)}{\blambda^2}\right)\sec^2\left(\frac{\pi}{2}r\right)+\frac{\widetilde{u}_{\blambda}(r)}{\blambda^2}\,\,\,\,\,\mbox{and}\,\,\,
\sup_{K}|\widetilde{u}_{\blambda}|\stackrel{\blambda\to\infty}{-\!\!\!-\!\!\!\longrightarrow}0,
\end{align} 
where $U$ defined in \eqref{eq:1.13} is the unique solution to \eqref{eq:1.12}.
In contrast to the asymptotics of $u_{\blambda}$ in any compact subset of $[0,1)$, our near-field analysis reveals a totally different asymptotic behaviour of $u_{\blambda}$ near the boundary $r=1$.
In Theorem~\ref{thm:2.2}, we show that $u_{\blambda}$ asymptotically blows up near the boundary $r=1$ and obtain a novel asymptotics 
\begin{equation}\label{eq:1.17}
u_{\blambda}\left(\rpal\right)={\min\{2,2\alpha\}}\log\frac{1}{\blambda}+\mathcal{O}_{p,\alpha},
\end{equation}
where $\mathcal{O}_{p,\alpha}$ depends mainly on $p$ and $\alpha$ and satisfies
\begin{equation*}
\limsup_{\blambda\to\infty}|\mathcal{O}_{p,\alpha}|<\infty
\end{equation*}
(see Theorem~\ref{thm:2.2} for the detailed expression of $\mathcal{O}_{p,\alpha}$).
As it was mentioned previously, we shall stress that the concept of the near-field expansions focus on the pointwise asymptotic behavior of solutions sufficiently near the boundary, which is different from the standard matched inner solution.
As a consequence, by \eqref{eq:1.16} and \eqref{eq:1.17} we know that a strong change of $u_{\blambda}$ with $\blambda\to\infty$ merely occurs near the boundary $r=1$, and there hold: (i) as $0<\alpha<1$, $\displaystyle\lim_{\blambda\to\infty}\textstyle\frac{u_{\blambda}(\rpal)}{\log\blambda^{-1}}=2\alpha$ shows that the blow-up asymptotics of $u_{\blambda}(\rpal)$ varies with $\alpha$; (ii) for $0<\alpha_1<\alpha_2<1$, interior points $1-{\blambda^{-\alpha_i}}$, $i=1,2$, are sufficiently close to the boundary, but the corresponding potential difference $|u_{\blambda}(1-{\blambda^{-\alpha_1}})-u_{\blambda}(1-{\blambda^{-\alpha_2}})|$ tends to infinity as $\blambda\gg1$. Furthermore, we show the following properties that emphasize the significant change of $u_{\blambda}$ near the boundary $r=1$ as $\blambda\gg1$ (which can be obtained from \eqref{eq:2.06}, \eqref{eq:2.08} and \eqref{eq:2.10}):
\begin{itemize}
\item[\textbf{(A)}]  For $\rpal$ satisfying \eqref{eq:1.14} with $0<\alpha\leq1$, \textbf{the second order term} of $u_{\blambda}(\rpal)$ (i.e., the leading order term of $\mathcal{O}_{p,\alpha}$ as $\blambda\gg1$) relies exactly on $p$; however, when $\alpha>1$, the second order term of $u_{\blambda}(\rpal)$ is independent of $p$.
\item[\textbf{(B)}] Various potential differences $|u_{\blambda}(\rpal)-u_{\blambda}(1)|$ and $|u_{\blambda}\left(r_{p_1;\alpha}^{\blambda}\right)-u_{\blambda}\left(r_{p_2;\alpha}^{\blambda}\right)|$ with $p_1\neq{p}_2$ can be presented as follows: 
\begin{align}
\label{eq:1.18}
\lim_{\blambda\to\infty}\left|u_{\blambda}(\rpal)-u_{\blambda}\left(1\right)\right|=
\begin{cases}
\infty,&\mbox{if}\,\,0<\alpha<1,\\
\displaystyle2\log\frac{p+2}{2},&\mbox{if}\,\,\alpha=1,\\0,&\mbox{if}\,\,\alpha>1,
\end{cases}
\end{align}
and 
\begin{align}
\label{eq:1.19}
\lim_{\blambda\to\infty}\left|u_{\blambda}\left(r_{p_1;\alpha}^{\blambda}\right)-u_{\blambda}\left(r_{p_2;\alpha}^{\blambda}\right)\right|=
\begin{cases}\displaystyle2\left|\log\frac{p_1}{p_2}\right|,&\mbox{if}\,\,0<\alpha<1,
\\[3mm]\displaystyle2\left|\log\frac{p_1+2}{p_2+2}\right|\in(0,2\left|\log\frac{p_1}{p_2}\right|),&\mbox{if}\,\,\alpha=1,\\
0,&\mbox{if}\,\,\alpha>1.
\end{cases}
\end{align}
Since $\mbox{dist}(\rpal,1)\to0$ and $\mbox{dist}(r_{p_1;\alpha}^{\blambda},r_{p_2;\alpha}^{\blambda})\to0$ as $\blambda\to\infty$, \eqref{eq:1.18} and \eqref{eq:1.19} present that the potential $u_{\blambda}$ has a sharp change in a quite thin region next to the boundary $r=1$. We refer the reader to Theorem~\ref{thm:2.2} for the asymptotic expansions (at most the exact first three order terms) of $u_{\blambda}(\rpal)$ and $u_{\blambda}'(\rpal)$ with respect to  $\blambda\gg1$.
\end{itemize}

\begin{itemize}
\item[\textbf{(C)}]  Based on \eqref{eq:1.16} and \eqref{eq:1.17}, the far-field and near-field expansions of $\rho_{\blambda}$ with respect to $\blambda\gg1$ are established.
Moreover, we show in Corollary~\ref{coro:2.3} that as $\blambda\to\infty$, the net charge density $\frac{e^{-u_{\blambda}}}{\int_0^1e^{-u_{\blambda}(s)}\text{d}s}=\frac{\rho_{\blambda}}{\blambda}$ behaves exactly as a Dirac measure supported at boundary point $r=1$.
This mathematically confirms that the non-neutral phenomenon occurs near the charged surface. \end{itemize}

\subsection{A new comparison with charge-conserving Poisson--Boltzmann equations}
\label{subsec:1.2}

In this section, we pay attention to a bi-nonlocal Poisson--Boltzmann equation for monovalent binary
ions (usually called the charge-conserving Poisson--Boltzmann equation~\cite{WXLS2014})
\begin{align}
\label{eq:1.20}
&v_{\bmu,\blambda}''(r)=\frac{\bmu e^{v_{\bmu,\blambda}(r)}}{\int_0^1e^{v_{\bmu,\blambda}(s)}\mathrm{d}s}-\frac{\blambda e^{-v_{\bmu,\blambda}(r)}}{\int_0^1e^{-v_{\bmu,\blambda}(s)}\mathrm{d}s},~~r\in(0,1),
\end{align}
with the same boundary condition of $u_{\blambda}$ as \eqref{eq:1.06}:
\begin{align}
\label{eq:1.21}
v_{\bmu,\blambda}(0)=v_{\bmu,\blambda}'(0)=0.
\end{align}
Equation \eqref{eq:1.20} is derived from the steady-state of PNP equation for monovalent binary electrolytes (cf. \cite{Lee2016-1,Lee2019,Lee2011,Lee2015}), where $\bmu$ and $\blambda$ are positive parameters related to the total number of anions and cations, respectively.
When we take a formal look at the case $0<\bmu\ll\blambda$, i.e., the total number of cations is great larger than that of anions, it seems that equation \eqref{eq:1.20}--\eqref{eq:1.21} approaches equation~\eqref{eq:1.05}--\eqref{eq:1.06}.
However, since the rigorous asymptotic behavior of those nonlocal terms are unknown, it is not obvious that $0<\bmu\ll\blambda$ implies $\frac{\bmu}{\int_0^1e^{v_{\bmu,\blambda}(s)}\mathrm{d}s}\ll\frac{\blambda}{\int_0^1e^{-v_{\bmu,\blambda}(s)}\mathrm{d}s}$.
Hence, a question is naturally raised: 
\begin{itemize}
\item[\textbf{(Q)}] \textit{Assume that} $\bmu$ \textit{depends on} $\blambda$. \textit{What does the relation between} $\bmu$ \textit{and} $\blambda$ \textit{make}
\begin{align}
\label{eq:1.22}
\lim_{\blambda\to\infty}\|v_{\bmu,\blambda}-u_{\blambda}\|_{\mathcal{C}^1([0,1])}=0
\end{align}
\textit{hold? Here $\left\|\,\cdot\,\right\|_{\mathcal{C}([0,1])}$ is defined in \eqref{eq:1.15} with $K=[0,1]$.}
\end{itemize}

Let us first make a brief review on \eqref{eq:1.20}--\eqref{eq:1.21} and point out the difficulty in studying the question (Q).
It is known (cf. \cite{Lee2019,Lee2011}) that for the case
\begin{equation}
\label{eq:1.23}
\bmu=\gamma\blambda\stackrel{\blambda\to\infty}{-\!\!\!-\!\!\!-\!\!\!\rightarrow}\infty~~(\gamma>0\mbox{~independent of~}\blambda),
\end{equation}
there holds $\displaystyle\lim_{\blambda\to\infty}|v_{\bmu,\blambda}|=0$ exponentially in any compact subset $K$ of $[0,1)$.
Moreover, under \eqref{eq:1.23}, by following the similar arguments as in \cite[(2.23)]{Lee2019}, we have that $\frac{\bmu}{\int_0^1e^{v_{\bmu,\blambda}(s)}\mathrm{d}s}$ and $\frac{\blambda}{\int_0^1e^{-v_{\bmu,\blambda}(s)}\mathrm{d}s}$ are divergent as $\blambda\to\infty$ since
\[\frac{\bmu}{\blambda\int_0^1e^{v_{\bmu,\blambda}(s)}\mathrm{d}s}\to\gamma~~\mbox{and}~~\frac1{\int_0^1e^{-v_{\bmu,\blambda}(s)}\mathrm{d}s}\to\gamma~~\mbox{as}~\blambda\to\infty.
\]
In this case, the asymptotic behavior of $v_{\bmu,\blambda}$ is totally different from that of $u_{\blambda}$ since $\displaystyle\lim_{\blambda\to\infty}\sup_{K}|u_{\blambda}-U|=0$ (see \eqref{eq:1.16}) and $\displaystyle\lim_{\blambda\to\infty}\textstyle\frac1{\int_0^1e^{-u_{\blambda}(s)}\mathrm{d}s}=0$ (see \eqref{eq:2.01}).
As a consequence, \eqref{eq:1.22} never holds under the condition \eqref{eq:1.23}.
We shall stress that the study of (Q) is different from the case in most recent work \cite{Lee2019} since the main analysis technique in \cite{Lee2019} needs the constraint \eqref{eq:1.23}.
Because of the limitation of analysis technique in these literatures, as $\blambda\to\infty$, the asymptotic behavior of the nonlocal coefficient $\frac{\bmu}{\int_0^1e^{v_{\bmu,\blambda}(s)}\mathrm{d}s}$ and $v_{\bmu,\blambda}$ without assumption \eqref{eq:1.23} remains unknown.

We take an essential viewpoint to answer question (Q).
Thanks to the inverse H\"older type estimate established in \cite[(3.8)]{ccleerolf2018}, we have, for $\bmu>0$ and $\blambda>0$, the estimate $1\leq\int_0^1e^{v_{\bmu,\blambda}(s)}\mathrm{d}s\int_0^1e^{-v_{\bmu,\blambda}(s)}\mathrm{d}s\leq\max\{\frac{\blambda}{\bmu},\frac{\bmu}{\blambda}\}$.
This implies, for $\blambda>\bmu>0$, 
\begin{align}
\label{eq:1.25}
\bmu^2\leq\frac{\bmu}{\int_0^1e^{v_{\bmu,\blambda}(s)}\mathrm{d}s}\cdot\frac{\blambda}{\int_0^1e^{-v_{\bmu,\blambda}(s)}\mathrm{d}s}\leq\bmu\blambda.
\end{align}
When $\blambda>0$ is fixed and $\bmu\to0^+$, equation \eqref{eq:1.20}--\eqref{eq:1.21} of $v_{\bmu,\blambda}$ formally approaches equation \eqref{eq:1.05}--\eqref{eq:1.06} of $u_{\blambda}$ since \eqref{eq:1.25} implies $\displaystyle\lim_{\bmu\to0^+}\textstyle\frac{\bmu}{\int_0^1e^{v_{\bmu,\blambda}(s)}\mathrm{d}s}=0$.
However, when $\blambda\to\infty$ and $\bmu\to0^+$ independently, it is not intuitive  to claim $\frac{\bmu}{\int_0^1e^{v_{\bmu,\blambda}(s)}\mathrm{d}s}\to0$ because we do not have the further information about $\bmu\blambda$.
From another viewpoint, if we first assume that $\bmu$ depends on $\blambda$ and \eqref{eq:1.22} holds, then we have $\frac{{\blambda}}{\int_0^1e^{-v_{{\bmu},\blambda}(s)}\text{d}s}\sim\frac{{\blambda}}{\int_0^1e^{-u_{\blambda}(s)}\text{d}s}\to\frac{\pi^2}{2}$ as $\blambda\to\infty$ (cf. \eqref{eq:2.01}).
Along with \eqref{eq:1.25}, we find that the condition
\begin{align}
\label{eq:1.26}
\lim_{\blambda\to\infty}\bmu\blambda=0
\end{align}
verifies $\displaystyle\lim_{\blambda\to\infty}\textstyle\frac{\bmu}{\int_0^1e^{v_{\bmu,\blambda}(s)}\mbox{d}s}=0$, together with $v_{{\bmu},\blambda}(r)\leq{v}_{{\bmu},\blambda}(0)=0$ (cf. Lemma~\ref{lma:4.1}), we obtain  
\[
\max_{[0,1]}\frac{\bmu e^{v_{\bmu,\blambda}}}{\int_0^1e^{v_{\bmu,\blambda}(s)}\mathrm{d}s}=\frac{\bmu}{\int_0^1e^{v_{\bmu,\blambda}(s)}\mathrm{d}s}\to0\,\,\mathrm{as}\,\,\blambda\to\infty,
\]
and equation~\eqref{eq:1.20} formally approaches equation~\eqref{eq:1.05}.
The following theorem confirms such an observation and establishes convergence of $v_{\bmu,\blambda}$ with $\bmu\to0$.
In particular, \eqref{eq:1.26} is a sufficient condition for \eqref{eq:1.22}.

\begin{thm}
\label{thm:1.1}
For $\blambda>\bmu>0$, let $v_{\bmu,\blambda}\in\mathcal{C}^2([0,1])$ be the unique solution to \eqref{eq:1.20}--\eqref{eq:1.21} (cf. \cite{Lee2014,Lee2019}), and let $u_{\blambda}\in\mathcal{C}^2([0,1])$ be the unique solution to \eqref{eq:1.05}--\eqref{eq:1.06}. Then both $v_{\bmu,\blambda}$ and $u_{\blambda}$ are monotonically decreasing, but $v_{\bmu,\blambda}-u_{\blambda}$ is monotonically increasing. Moreover, we have:
\begin{enumerate}
\item[(a)] Assume \eqref{eq:1.26}. Then \eqref{eq:1.22} is achieved. 
\item[(b)] For $\blambda>0$ fixed, there holds 
\[
\lim_{\bmu\to0^+}\|v_{\bmu,\blambda}-u_{\blambda}\|_{\mathcal{C}^1([0,1])}=0.
\]
\end{enumerate}
\end{thm}

Here we provide a numerical result to support Theorem~\ref{thm:1.1}(b) in Figure~\ref{fig}.
\begin{figure}[h]\centering\includegraphics[scale=0.72]{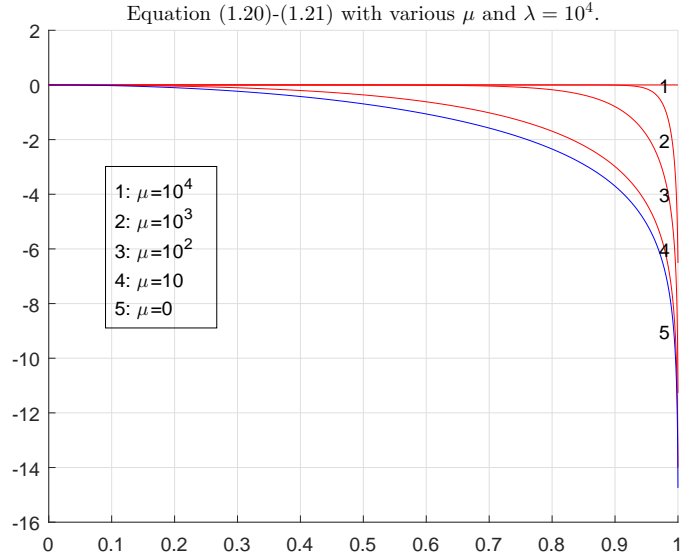}\caption{For $\blambda>0$ fixed, equation \eqref{eq:1.20}--\eqref{eq:1.21} of $v_{\bmu,\blambda}$ approximates equation \eqref{eq:1.05}--\eqref{eq:1.06} of $u_{\blambda}$ as $\bmu$ tends to zero. Given $\blambda=10^4$, the curve 1--5 are associated with $\bmu=10^4, 10^3, 10^2, 10, 0$, respectively.
Note that the curve 1 presents the neutrality ($\bmu=\blambda=10^4$), which implies $v_{\bmu,\blambda}\equiv0$.}
\label{fig}\end{figure}

Setting $w_{\bmu,\blambda}:=v_{\bmu,\blambda}-u_{\blambda}$, we obtain
\begin{align*}
w_{\bmu,\blambda}''\left(r\right)=e^{-v_{\bmu,\blambda}(r)}W_{\bmu,\blambda}(r)\,\,\text{with}\,\,W_{\bmu,\blambda}(r)=\frac{\bmu e^{2v_{{\bmu},\blambda}(r)}}{\int_0^1e^{v_{\bmu,\blambda}(s)}\mathrm{d}s}-\frac{\blambda}{\int_0^1e^{-v_{\bmu,\blambda}(s)}\mathrm{d}s}+\frac{\blambda e^{w_{\bmu,\blambda}(r)}}{\int_0^1e^{-u_{\blambda}(s)}\mathrm{d}s}.
\end{align*}
We shall briefly sketch the proof of Theorem~\ref{thm:1.1} as follows.
\begin{itemize}
\item The basic properties of $v_{\bmu,\blambda}$ and $w_{\blambda}$ will be stated in Section~\ref{subsec:4.1}. Furthermore, in Lemma~\ref{lma:4.2} we carefully deal with $W_{\bmu,\blambda}$ for the case $\blambda>\bmu>0$, and show that $w_{\bmu,\blambda}$ has neither local maximum nor local minimum in $(0,1)$. Along with $w_{\bmu,\blambda}'(1)=v_{\bmu,\blambda}'(1)-u_{\blambda}'(1)=\bmu>0$, this particularly shows the monotonic increase of $w_{\bmu,\blambda}$ with~$r$.
\item A significant idea for proving Theorem~\ref{thm:1.1}(a) and (b) is to establish the following estimates:
\begin{itemize}
\item[\textbf{(1)}] For $w_{\bmu,\blambda}$, we obtain
\begin{equation}
\label{eq:1.29}
0\leq\max_{\left[0,1\right]}{w_{\bmu,\blambda}}={w_{\bmu,\blambda}(1)}<\log\bigg(\left(1-\frac{\bmu}{\blambda}\right)^2+\bmu\blambda\underbrace{\left(2-\frac{\bmu}{\blambda}\right)\frac{e^{-u_{\blambda}(1)}}{\blambda^2}}_{\mathop{\text{uniformly\,\,bounded}}\limits_{\mbox{\scriptsize\text{as}\,\,$\blambda\to\infty$}}}\bigg)\,\,\underset{\bmu\blambda\to0}{\overset{\blambda\,\,\text{fixed\,\,or}\,\,\blambda\to\infty}{-\!\!\!-\!\!\!-\!\!\!-\!\!\!-\!\!\!-\!\!\!-\!\!\!-\!\!\!-\!\!\!-\!\!\!-\!\!\!\rightarrow}}0
\end{equation}
(cf. \eqref{eq:4.15}--\eqref{eq:4.16} for $\blambda\to\infty$; \eqref{eq:4.24} for $\blambda>0$ fixed). Here we shall stress that in \eqref{eq:1.29}, the condition~\eqref{eq:1.26} seems optimal since $\displaystyle\lim_{\blambda\to\infty}\textstyle\frac{e^{-u_{\blambda}(1)}}{\blambda^{2}}=\pi^{-2}$ (see \eqref{eq:2.05}).
\item[\textbf{(2)}] For $w_{\bmu,\blambda}'$, we obtain
\[
\begin{aligned}
\max_{\left[0,1\right]}w_{\bmu,\blambda}'^2:=w_{\bmu,\blambda}'^2(r_{\bmu,\blambda}^*)\leq&\,-\frac{\bmu(2e^{v_{\bmu,\blambda}(r_{\bmu,\blambda}^*)}-1)}{\int_0^1e^{v_{\bmu,\blambda}(s)}\mathrm{d}s}\\
&\,+\blambda\left(\frac1{\int_0^1e^{-v_{\bmu,\blambda}(s)}\mathrm{d}s}-\frac1{\int_0^1e^{-u_{\blambda}(s)}\mathrm{d}s}\right)\underset{\bmu\blambda\to0}{\overset{\blambda\,\,\text{fixed\,\,or}\,\,\blambda\to\infty}{-\!\!\!-\!\!\!-\!\!\!-\!\!\!-\!\!\!-\!\!\!-\!\!\!-\!\!\!-\!\!\!-\!\!\!-\!\!\!\rightarrow}}0
\end{aligned}
\]
(cf. \eqref{eq:4.21} and Lemma~\ref{lma:4.4}  for $\blambda\to\infty$; \eqref{eq:4.27}--\eqref{eq:4.28} for $\blambda>0$ fixed). 
\end{itemize}
\end{itemize}
Consequently, we obtain $\displaystyle\lim_{\mathop{\blambda\,\,\text{fixed\,\,or}\,\,\blambda\to\infty}\limits_{\mbox{\scriptsize $\bmu\blambda\to0$}}}\left\|w_{\bmu,\blambda}\right\|_{\mathcal{C}^1([0,1])}=0$. 
The detailed proof of Theorem~\ref{thm:1.1} will be stated in Sections~\ref{subsec:4.2}--\ref{subsec:4.4}.

In Theorems~\ref{thm:2.1}--\ref{thm:2.2} below, we will establish the refined far-field and near-field expansions for $u_{\blambda}$ with respect to $\blambda\gg1$.
Combining Theorem~\ref{thm:1.1}(a) with Theorems~\ref{thm:2.1}--\ref{thm:2.2}, we can obtain refined asymptotic expansions of $v_{\bmu,\blambda}$ as $\blambda\gg1$ and $0<\bmu\blambda\ll1$.
Various asymptotics of $v_{{\bmu},\blambda}$ and $u_{\blambda}$ can be presented as follows.

\setlength{\unitlength}{1cm}
\thicklines
\begin{picture}(17,3.3)
\put(5,2.3){\mbox{$v_{{\bmu},\blambda}$}}
\put(6,2.3){\vector(1,0){3}}
\put(9.1,2.3){\mbox{$u_{\blambda}$}}
\put(6.8,2.5){\text{\small (uniformly)}}
\put(6.1,1.9){\text{\small $\blambda$ fixed and ${\bmu}\to0$}}
\put(10,2.5){\text{\small (pointwise)}}
\put(12,2.2){\text{$U:=\displaystyle\lim_{\blambda\to\infty}u_{\blambda}$}}
\put(9.8,2.3){\vector(1,0){2}}
\put(10.3,1.9){\text{\small $\blambda\to\infty$}}

\put(5,1){\mbox{$v_{{\bmu},\blambda}-u_{\blambda}$}}
\put(6.6,1){\vector(1,0){3.5}}
\put(10.2,0.9){\mbox{$0$}}
\put(7.5,1.2){\text{\small (uniformly)}}
\put(6.9,0.5){\text{\small $\blambda\to\infty$ and $\bmu\blambda\to0$}}
\end{picture}

{\bf Organization of the paper.}
The rest of the paper is organized as follows. In Section~\ref{sec2} we will state the main results about the far-field expansions and the near-field expansions of $u_{\blambda}$ in Theorems~\ref{thm:2.1} and \ref{thm:2.2}, respectively. Based on such asymptotic expansions, we establish in Corollary~\ref{coro:2.3} for the refined asymptotic expansions of $\rho_{\blambda}$ and the related concentration phenomenon of $\frac{\rho_{\blambda}}{\blambda}$ as $\blambda\to\infty$. Afterwards, we prove Theorems~\ref{thm:2.1}--\ref{thm:2.2} and Corollary~\ref{coro:2.3} in Section~\ref{sec3}. We shall stress that Theorem~\ref{thm:2.2} plays a crucial role in the proof of Theorem~\ref{thm:1.1}. In Section~\ref{subsec:4.1} we introduce some basic properties of $v_{\bmu,\blambda}$. For Theorem~\ref{thm:1.1}, we will state the proof of (a) in Sections~\ref{subsec:4.2}--\ref{subsec:4.3} and the proof of (b) in Section~\ref{subsec:4.4}. Finally, we provide an application to calculating the capacitances and discuss such result in Section~\ref{sec5}.

\section{The main results of \texorpdfstring{\eqref{eq:1.05}}{(1.5)}--\texorpdfstring{\eqref{eq:1.06}}{(1.6)}}
\label{sec2}

Throughout the whole paper, we denote $o_{\blambda}$ as the quantity tending to zero as $\blambda$ goes to infinity. We are now in a position to state the main results about the far-field and near-field expansions of the solution to \eqref{eq:1.05}--\eqref{eq:1.06} as follows.

\begin{thm}[Far-field expansions of $u_{\blambda}$]
\label{thm:2.1}
For $\blambda>0$, let $u_{\blambda}\in \mathcal{C}^2([0,1])$ be the unique solution to \eqref{eq:1.05}--\eqref{eq:1.06}. Then as $\blambda\gg1$,
\begin{equation}
\label{eq:2.01}
\int_0^1e^{-u_{\blambda}(s)}\mathrm{d}s=\frac{2}{\pi^2}\left(\blambda+4+\frac{4}\blambda\left(1+o_{\blambda}\right)\right)\to\infty.
\end{equation}
Moreover, for any compact subset $K$ of $[0,1)$, there holds
\begin{align}
\label{eq:2.02}
\blambda^2\left\|u_{\blambda}(r)-U(r)-\pi{r}\left(\frac{\sin({\pi}r)}{\blambda}-\frac{\pi{r}+2\sin({\pi}r)}{\blambda^2}\right)\sec^2\left(\frac{\pi}{2}r\right)\right\|_{\mathcal{C}^1(K)}\stackrel{\blambda\to\infty}{-\!\!\!-\!\!\!-\!\!\!\rightarrow}0,
\end{align}
where $U(r)=2\log\cos\left(\frac{\pi}{2}r\right)$ is the unique solution to \eqref{eq:1.12} and $\left\|\,\cdot\,\right\|_{\mathcal{C}^1(K)}$ is defined in \eqref{eq:1.15}.
\end{thm}

\begin{remark}\label{remark1}
\eqref{eq:2.02} gives the precise first three terms of $u_{\blambda}(r)$ and $u_{\blambda}'(r)$ with respect to $\blambda\gg1$, for $r\in[0,1)$:
\begin{equation}\label{eq:2.03}
\begin{aligned}
&u_{\blambda}(r)=U(r)+\pi r\left(\frac{\sin(\pi r)}{\blambda}-\frac{\pi r+2\sin(\pi r)}{\blambda^2}\right)\sec^2\left(\frac\pi2r\right)+O\left(\frac1{\blambda^3}\right),
\\&
u'_{\blambda}=U'(r)+\frac{2\pi\tan(\pi r/2)+\pi^2r\sec^2(\pi r/2)}{\blambda}\\&\qquad-\frac{4\pi\tan(\pi r/2)+4\pi^2r\sec^2(\pi r/2)+\pi^3r^2\sec^2(\pi r/2)\tan(\pi r/2)}{\blambda^2}+O\left(\frac1{\blambda^3}\right)\end{aligned}
\end{equation}
In particular, one obtains
\begin{equation}
\label{eq:2.04}
\lim_{\blambda\to\infty}u_{\blambda}(r)=U(r)\text{  and   }~\lim_{\blambda\to\infty}u_{\blambda}'(r)=U'(r)=-\pi\tan\left(\frac{\pi}{2}r\right).
\end{equation}
We stress that \eqref{eq:2.04} can be derived directly from \cite{Cartailler2017}. However, the argument in \cite{Cartailler2017} seems difficult to establish \eqref{eq:2.03} due to the lack of \eqref{eq:3.03}.
\end{remark}

\begin{thm}[Near-field expansions of $u_{\blambda}$]\label{thm:2.2}
Under the same hypotheses as in Theorem~\ref{thm:2.1}, as $\blambda\gg1$, we have
\begin{equation}\label{eq:2.05}
u_{\blambda}(1)=2\log\frac{1}{\blambda}+2\log\pi-\frac{4}{\blambda}(1+o_{\blambda}).
\end{equation}
Moreover, for $\alpha,~p>0$ independent of $\blambda$, $u_{\blambda}(\rpal)$ asymptotically blows up, which is depicted as follows:
\begin{enumerate}
\item[(a)] If $\alpha>2$, i.e., $\blambda^2\mathrm{dist}(\rpal,1)\to0$, then $u_{\blambda}\left(\rpal\right)=-2\log\blambda+2\log\pi-\frac{4}{\blambda}(1+o_{\blambda})$, which shares the same first three terms with $u_{\blambda}(1)$, and $u_{\blambda}'(\rpal)$ and $u_{\blambda}'(1)$ shares the same leading order term, which asymptotically blows up.
In particular, $\left|u_{\blambda}'(\rpal)-u_{\blambda}'(1)\right|\to0$ as $\blambda\to\infty$.
\item[(b)] If $1<\alpha\leq2$, then
\begin{enumerate}
\item[(b1)] $u_{\blambda}(\rpal)$ and $u_{\blambda}(1)$ share the same first two terms:
\begin{equation}
\label{eq:2.06}
u_{\blambda}(\rpal)=2\log\frac{1}{\blambda}+2\log\pi+\begin{cases}
\displaystyle\frac{p-4}{\blambda}\left(1+o_{\blambda}\right)&\mbox{if }\alpha=2,\\[3mm]
\displaystyle\frac{p}{\blambda^{\alpha-1}}\left(1+o_{\blambda}\right)&\mbox{if }\alpha\in\left(1,2\right).\end{cases}
\end{equation}
\item[(b2)] $u_{\blambda}'(\rpal)$ and $u_{\blambda}'(1)$ share the same leading order term:
\begin{equation}
\label{eq:2.07}
u_{\blambda}'(\rpal)=-\blambda+\frac{p}{2}\blambda^{2-\alpha}(1+o_{\blambda}),\quad\alpha\in(1,2].
\end{equation}
\end{enumerate}
Hence, the effect of $p$ and $\alpha$ occurs at the third order term of $u_{\blambda}(\rpal)$ and the second order term $u_{\blambda}'(\rpal)$.
\item[(c)] If $\alpha=1$, then
\begin{enumerate}
\item[(c1)] $u_{\blambda}(\rpal)$ and $u_{\blambda}(1)$ share the same leading order terms:
\begin{equation}
\label{eq:2.08}
u_{\blambda}(\rpal)=2\log\frac{1}{\blambda}+2\log\frac{\left(p+2\right)\pi}{2}-\frac{4}{\blambda}\left(1+o_{\blambda}\right).
\end{equation}
\item[(c2)] The leading order term of $u_{\blambda}'(\rpal)$ depends on $p$:
\begin{equation}
\label{eq:2.09}
u_{\blambda}'(\rpal)=-\frac{2}{p+2}\blambda+o_{\blambda}.
\end{equation}
\end{enumerate}
Hence, the effect of $p$ occurs at the second order term of $u_{\blambda}(\rpal)$ and the leading order term of $u_{\blambda}'(\rpal)$.
\item[(d)] If $0<\alpha<1$, then
\begin{enumerate}
\item[(d1)] The leading order term of $u_{\blambda}(\rpal)$ depends on $\alpha$ and second order term depends on $p$:
\begin{equation}
\label{eq:2.10}
u_{\blambda}(\rpal)
={2\alpha}\log\frac{1}{\blambda}+
2\log\frac{p\pi}{2}+o_{\blambda}.
\end{equation}
\item[(d2)] The leading order term of $u_{\blambda}'(\rpal)$ depends on $p$ and $\alpha$:
\begin{equation}
\label{eq:2.11}
u_{\blambda}'(\rpal)=-\frac{2}{p}\blambda^{\alpha}+o_{\blambda}.
\end{equation}
\end{enumerate}
\end{enumerate}
\end{thm}

Thanks to Theorems~\ref{thm:2.1} and \ref{thm:2.2}, we are able to obtain the pointwise asymptotics of $\blambda^{-1}\rho_{\blambda}$.
Moreover, we have  (cf. \eqref{eq:2.12} and \eqref{eq:2.13}) that as $\blambda\gg1$:
\begin{itemize}
\item For $0<\alpha<1/2$, $\displaystyle\sup_{[0,\rpal]}\blambda^{-1}\rho_{\blambda}$ and $\displaystyle\sup_{[0,\rpal]}\blambda^{-1}u_{\blambda}'^2$ tend to zero.
\item For $\alpha=1/2$, $\blambda^{-1}\rho_{\blambda}(\rpal)$ and $\blambda^{-1}u_{\blambda}'^2(\rpal)$ are bounded and have positive lower bound.
\item For $\alpha>1/2$, $\blambda^{-1}\rho_{\blambda}(\rpal)$ and $\blambda^{-1}u_{\blambda}'^2(\rpal)$ asymptotically blow up.
\end{itemize}
More precisely, we obtain the refined far-field and near-field expansions of the net charge density $\rho_{\blambda}$ and the boundary concentration phenomena of $\blambda^{-1}\rho_{\blambda}$ and  $\blambda^{-1}u_{\blambda}'^2$, which are stated as follows.

\begin{coro}\label{coro:2.3} 
Under the same hypotheses as in Theorem~\ref{thm:2.1}, as $\blambda$ approaches infinity, we have 
\begin{enumerate}
\item[(a)] (Far-field expansions of $\rho_{\blambda}$) For any compact subset $K$  (independent of $\blambda$) of $[0,1)$, 
\begin{equation}
\label{eq:2.12}
\rho_{\blambda}(r)=\frac{\pi^2}2\sec^2\left(\frac{\pi}2r\right)-\frac{1}{\blambda}\left[2\pi^2\sec^2\left(\frac{\pi }{2}r\right)+\pi^3\sec^2\left(\frac{\pi }{2}r\right)\tan\left(\frac{\pi }{2}r\right)\right](1+o_{\blambda})
\end{equation}
uniformly in $K$.
\item[(b)] (Near-field expansions of $\rho_{\blambda}$) The asymptotics of $\rho_{\blambda}$ near the boundary is depicted as follows:
\begin{equation}
\label{eq:2.13}
\rho_{\blambda}(\rpal)=
\begin{cases}
\displaystyle\frac{\blambda^2}{2}(1+o_{\blambda})&\mbox{if }\alpha\in(1,\infty),\\[3mm]\displaystyle\frac{2\blambda^2}{(p+2)^2}(1+o_{\blambda})&\mbox{if }\alpha=1,\\[3mm]
\displaystyle
\frac{2\blambda^{2\alpha}}{p^2}(1+o_{\blambda})&\mbox{if }\alpha\in(0,1).
\end{cases}
\end{equation}
\item[(c)] (Concentration phenomenon) Both $\blambda^{-1}\rho_{\blambda}$ and $(2\blambda)^{-1}u_{\blambda}'^2$ behave exactly as Dirac measures supported at boundary point $r=1$, i.e.,
\begin{align}
\label{eq:2.14}
\lim_{\blambda\to\infty}\int_0^1\frac{\rho_{\blambda}(r)}{\blambda}h(r)\mathrm{d}r&=h(1),
\\\label{eq:2.15}
\lim_{\blambda\to\infty}\int_0^1\frac{u_{\blambda}'^2(r)}{2\blambda}h(r)\mathrm{d}r&=h(1),
\end{align}
for any continuous function $h:[0,1]\to\mathbb{R}$ independent of $\blambda$.
\end{enumerate}
\end{coro}

\section{Proof of Theorems~\ref{thm:2.1}--\ref{thm:2.2} and Corollary~\ref{coro:2.3}}
\label{sec3}

It is well-known (cf. \cite[Appendix]{Cartailler2017}) that $u_{\blambda}$ and $\rho_{\blambda}$ can be expressed as
\begin{equation}
\label{eq:3.01}
u_{\blambda}(r)=2\log\cos\left(\sqrt{\frac{J_{\blambda}}2}r\right)~~\mbox{and}~~\rho_{\blambda}(r)=J_{\blambda}\sec^2\left(\sqrt{\frac{J_{\blambda}}{2}}r\right)~~\text{for}~r\in\left[0,1\right],
\end{equation}
where
\begin{equation}
\label{eq:3.02}
J_{\blambda}=\frac{\blambda}{I_{\blambda}}<\frac{\pi^2}{2},\,\,\,\,\text{and}\,\,\,\,I_{\blambda}=\int_0^1e^{-u_{\blambda}(s)}\mathrm{d}s.
\end{equation}

To study the asymptotic behaviour of $u_{\blambda}$, it suffices to establish the refined asymptotic expansions of $I_{\blambda}$ and $J_{\blambda}$ which are stated as follows.
\begin{lma}
\label{lma:3.1}
Under the same hypothesis in Theorem~\ref{thm:2.1}, we have
\begin{equation}
\label{eq:3.03}
J_{\blambda}=\pi^2\left(\frac{1}2-\frac{2}\blambda+\frac{6}{\blambda^2}-\frac{48-2\pi^2}{3\blambda^3}(1+o_{\blambda})\right)\,\,as\,\,\blambda\gg1.\end{equation}
\end{lma}
The proof of Lemma~\ref{lma:3.1} is elementary so we state it in Appendix.

\begin{remark}\label{rk-20-1111}
Since \eqref{eq:1.05} has a unique solution $u_{\blambda}$ and corresponds to the nonlinear eigenvalue problem \eqref{eq:1.10} with $(U_{\kappa},\kappa)=(u_{\blambda},J_{\blambda})$, this results in $\frac{\blambda}{\int_0^1e^{-u_{\blambda}(s)}\mathrm{d}s}<\frac{\pi^2}{2}$ for any positive $\blambda$. We want to point out that \eqref{eq:3.03} implies a squeezed estimate
\[\pi^2\left(\frac{1}2-\frac{2}\blambda+\frac{6-\epsilon}{\blambda^2}\right)<\frac{\blambda}{\int_0^1e^{-u_{\blambda}(s)}\mathrm{d}s}<\pi^2\left(\frac{1}2-\frac{2}\blambda+\frac{6}{\blambda^2}\right)\,\,as\,\, \blambda>\blambda(\epsilon),
\]
for any sufficiently small  $\epsilon>0$, where $\blambda(\epsilon)$ is a positive constant  depending on $\epsilon$.
\end{remark}

Now, we are in a position to state the proof of Theorem~\ref{thm:2.1}.
\begin{proof}[$\mathbf{Proof~of~ Theorem~\ref{thm:2.1}}$]

\eqref{eq:2.01} follows from \eqref{eq:3.03} since we have
\[
I_{\blambda}=\frac{\blambda}{J_{\blambda}}=\frac{2\blambda}{\pi^2}\left(1-\frac{4}\blambda+\frac{12}{\blambda^2}(1+o_{\boldsymbol{\lambda}})\right)^{-1}=\frac{2}{\pi^2}\left(\blambda+4+\frac{4}\blambda(1+o_{\blambda})\right).
\]
To prove \eqref{eq:2.02}, we need to  establish the precise first third order terms of $u_{\blambda}(r)$ and $u_{\blambda}'(r)$ in $K$.
Firstly, by \eqref{eq:3.01} and \eqref{eq:3.03}, there holds that 
\begin{equation}
\label{eq:3.06}
u_{\blambda}(r)=2\log\cos\left(\frac{\pi}2r-\frac{\pi}{\blambda}r+\frac{2\pi}{\blambda^2}r(1+o_{\blambda})\right)\,\,\mathrm{uniformly\,\,in}\,\,K,\,\,\mathrm{as}\,\,\blambda\to\infty.\end{equation}
Note that $r\in K$ is independent of $\blambda$.
For a sake of convenience, let us set 
\begin{equation}
\label{eq:3.07}
\xi(r)=-\frac{\pi}{\blambda}r+\frac{2\pi}{\blambda^2}r(1+o_{\blambda}).    
\end{equation}
Then, as $\blambda\to\infty$, we have $\xi(r)\to0$ and
\begin{equation}
\label{eq:3.08}
2\log\cos\left(\frac{\pi}{2}r+\xi(r)\right)=2\log\cos\left(\frac{\pi}{2}r\right)-2\xi(r)\tan\left(\frac{\pi}{2}r\right)-\xi^2(r)\sec^2\left(\frac{\pi}{2}r\right)+o(\xi^3(r)),
\end{equation}
uniformly in $K$.
\eqref{eq:3.08} can be obtained directly from the Taylor expansions so we omit the detailed derivation.
As a consequence, by \eqref{eq:3.06}--\eqref{eq:3.08}, 
\begin{equation}
\label{eq:3.09}\begin{aligned}
u_{\blambda}(r)=&2\log\cos\left(\frac{\pi}{2}r\right)+\frac{2\pi r}{\blambda}\tan\left(\frac{\pi}{2}r\right)-\frac{\pi r}{\blambda^2}\left(4\tan\left(\frac{\pi}{2}r\right)+\pi r\sec^2\left(\frac{\pi}{2}r\right)\right)(1+o_{\blambda})
\\=&2\log\cos\left(\frac{\pi}{2}r\right)+\pi r\sec^2\left(\frac{\pi}{2}r\right)\left(\frac{\sin(\pi r)}{\blambda}-\frac{\pi r+2\sin(\pi r)}{\blambda^2}\right)(1+o_{\blambda}),
\end{aligned}
\end{equation}
which gives the precise first third order terms of $u_{\blambda}(r)$.

On the other hand, differentiating \eqref{eq:3.01} to $r$ gives
\begin{equation}
\label{eq:3.10}u_{\blambda}'(r)=-2\sqrt{\frac{J_{\blambda}}2}\tan\left(\sqrt{\frac{J_{\blambda}}2}r\right).
\end{equation}
Hence, by virtue of \eqref{eq:3.03}, \eqref{eq:3.07} and \eqref{eq:3.10}, an expansion of $u_{\blambda}'(r)$ with respect to $\blambda\gg1$ can be expressed as 
\begin{equation}
\label{eq:3.11}
\begin{aligned}
u_{\blambda}'(r)=&\left(-\pi+\frac{2\pi}{\blambda}-\frac{4\pi}{\blambda^2}(1+o_{\blambda})\right)\tan\left(\frac{\pi}{2}r-\frac{\pi}{\blambda}r+\frac{2\pi}{\blambda^2}r(1+o_{\blambda})\right)
\\=&-\pi\tan\left(\frac{\pi}{2}r\right)+\frac{\pi^2r+\pi\sin(\pi r)}{\blambda}\sec^2\left(\frac{\pi}{2}r\right)
\\&-\frac{\pi}{\blambda^2}\left[4\tan\left(\frac{\pi}{2}r\right)+4\pi r\sec^2\left(\frac{\pi}{2}r\right)+\pi^2r^2\sec^2\left(\frac{\pi}{2}r\right)\tan\left(\frac{\pi}{2}r\right)\right](1+o_{\blambda}),
\end{aligned}
\end{equation}
uniformly in $K$.
Here we have used the approximation
\[\tan\left(\frac{\pi}{2}r+\xi(r)\right)=\tan\left(\frac{\pi}{2}r\right)+\xi(r)\sec^2\left(\frac{\pi}{2}r\right)+\xi^2(r)\sec^2\left(\frac{\pi}{2}r\right)\tan\left(\frac{\pi}{2}r\right)+o(\xi^3(r)).\]
Therefore, \eqref{eq:2.02} follows immediately from \eqref{eq:3.09} and  \eqref{eq:3.11}, and we complete the proof of Theorem~\ref{thm:2.1}.
\end{proof}

\subsection{Proof of Theorem~\texorpdfstring{\ref{thm:2.2}}{2.2}}
\label{subsec:3.1}

\eqref{eq:2.05} follows from \eqref{eq:1.09} and \eqref{eq:3.02}--\eqref{eq:3.03}.
By \eqref{eq:3.01} and \eqref{eq:3.10}, we have
\begin{equation}
\label{eq:3.12}
u_{\blambda}(\rpal)=2\log\cos\left(\rpal\sqrt{\frac{J_{\blambda}}2}\right)=2\log\sin\left(\frac{\pi}2-\rpal\sqrt{\frac{J_{\blambda}}2}\right),
\end{equation}
and
\begin{equation}
\label{eq:3.13}
u_{\blambda}'(\rpal)=-\sqrt{2J_{\blambda}}\tan\left(\rpal\sqrt{\frac{J_{\blambda}}2}\right),
\end{equation}
where $\rpal$ is defined in \eqref{eq:1.14}.
On the other hand, 
by \eqref{eq:1.14} and \eqref{eq:3.03}, we have
\begin{equation}
\label{eq:3.14}
\rpal\sqrt{\frac{J_{\blambda}}2}=\left(1-\frac{p}{\blambda^\alpha}\right)\left(\frac{\pi}{2}-\frac{\pi}{\blambda}+\frac{2\pi}{\blambda^2}(1+o_{\blambda})\right).
\end{equation}
Note that the asymptotic expansion of $\rpal\sqrt{J_{\blambda}/2}$ varies with $\alpha$.
To obtain the refined asymptotic expansions of \eqref{eq:3.12} and \eqref{eq:3.13}, we shall deal with the asymptotics of $\rpal\sqrt{J_{\blambda}/2}$ under situations $\alpha>2$, $1<\alpha\leq2$, $\alpha=1$ and $0<\alpha<1$ individually.

\subsubsection*{\textbf{Case 1}. $\alpha>2$:}

In this case, we have $0<\frac{1}{\blambda^{\alpha}}\ll\frac{1}{\blambda^2}$ as $\blambda\gg1$.
Hence,
\begin{equation}
\label{eq:3.15}
\rpal\sqrt{\frac{J_{\blambda}}2}=\left(1-\frac{p}{\blambda^\alpha}\right)\left(\frac{\pi}2-\frac{\pi}\blambda+\frac{2\pi}{\blambda^2}(1+o_{\blambda})\right)=\frac{\pi}2-\frac{\pi}\blambda+\frac{2\pi}{\blambda^2}(1+o_{\blambda}).
\end{equation}
Along with \eqref{eq:3.12}, one may make appropriate manipulations to obtain
\begin{equation}
\label{eq:3.16}
u_{\blambda}\left(\rpal\right)=\,2\log\sin\left(\frac{\pi}{\blambda}-\frac{2\pi}{\blambda^2}(1+o_{\blambda})\right)=\log\frac{1}{\blambda^2}+\log\pi^2-\frac{4}{\blambda}(1+o_{\blambda}).
\end{equation}
On the other hand, by \eqref{eq:3.13}, \eqref{eq:3.15} and Taylor expansions of the cotangent function, one may check that
\begin{equation}
\label{eq:3.17}
\begin{aligned}
u_{\blambda}'(\rpal)=&\left(-\pi+\frac{2\pi}{\blambda}-\frac{4\pi}{\blambda^2}\left(1+o_{\blambda}\right)\right)\tan\left[\left(1-\frac{p}{\blambda^\alpha}\right)\left(\frac{\pi}{2}-\frac{\pi}{\blambda}+\frac{2\pi}{\blambda^2}\left(1+o_{\blambda}\right)\right)\right]
\\=&\left(-\pi+\frac{2\pi}{\blambda}-\frac{4\pi}{\blambda^2}\left(1+o_{\blambda}\right)\right)\cot\left(\frac{\pi}{\blambda}-\frac{2\pi}{\blambda^2}\left(1+o_{\blambda}\right)\right)=-\blambda+o_{\blambda}.
\end{aligned}
\end{equation}
Therefore, Theorem~\ref{thm:2.2}(a) follows from \eqref{eq:3.16} and \eqref{eq:3.17}.

\subsubsection*{Case 2. $1<\alpha\leq2$:}

In this case, \eqref{eq:3.14} gives
\begin{equation}
\label{eq:3.18}
\rpal\sqrt{\frac{J_{\blambda}}2}=\left(1-\frac{p}{\blambda^\alpha}\right)\left(\frac{\pi}2-\frac{\pi}\blambda+\frac{2\pi}{\blambda^2}(1+o_{\blambda})\right)=\frac{\pi}2-\frac{\pi}\blambda-\frac{p\pi}{2\blambda^\alpha}+\frac{2\pi}{\blambda^2}(1+o_{\blambda}).
\end{equation}
Along with \eqref{eq:3.12} one obtains 
\[\begin{aligned}
u_{\blambda}(\rpal)=&2\log\left(\frac{\pi}\blambda+\frac{p\pi}{2\blambda^{\alpha}}-\frac{2\pi}{\blambda^2}(1+o_{\blambda})\right)=2\log\frac{\pi}\blambda+2\log\left(1+\frac{p}{2\blambda^{\alpha-1}}-\frac{2}\blambda(1+o_{\blambda})\right)
\\=&\log\frac{1}{\blambda^2}+\log\pi^2+\begin{cases}
\displaystyle\frac{p-4}{\blambda}(1+o_{\blambda})&\mbox{if }\alpha=2,\\[3mm]
\displaystyle\frac{p}{\blambda^{\alpha-1}}(1+o_{\blambda})&\mbox{if }\alpha\in\left(1,2\right),
\end{cases}
\end{aligned}
\]
and \eqref{eq:2.06} immediately follows from this result.
Combining \eqref{eq:3.13} with \eqref{eq:3.18}
and following similar argument, we get \eqref{eq:2.07} and complete the proof of Theorem~\ref{thm:2.2}(b).

\subsubsection*{Case 3. $\alpha=1$:}

In this case, \eqref{eq:3.14} becomes
\begin{equation}
\label{eq:3.20}
\rpal\sqrt{\frac{J_{\blambda}}2}=\left(1-\frac{p}\blambda\right)\left(\frac{\pi}2-\frac{\pi}\blambda+\frac{2\pi}{\blambda^2}(1+o_{\blambda})\right)=\frac{\pi}2-\frac{(p+2)\pi}{2\blambda}+\frac{(p+2)\pi}{\blambda^2}(1+o_{\blambda}).
\end{equation}
Hence, by inserting \eqref{eq:3.20} into \eqref{eq:3.12} and following the similar argument as in \eqref{eq:3.16}, we obtain
\[
u_{\blambda}(\rpal)=2\log\sin\left(\frac{(p+2)\pi}{2\blambda}-\frac{(p+2)\pi}{\blambda^2}(1+o_{\blambda})\right)=\log\frac{1}{\blambda^2}+\log\frac{(p+2)^2\pi^2}4-\frac{4}\blambda(1+o_{\blambda}).
\]
This implies \eqref{eq:2.08}.
Finally, \eqref{eq:2.09} can be obtained from \eqref{eq:3.13} and \eqref{eq:3.20}.
The proof of Theorem~\ref{thm:2.2}(c) is complete.

\subsubsection*{Case 4. $0<\alpha<1$:}

We want to emphasize that due to $0<\frac{1}{\blambda}\ll\frac{1}{\blambda^{\alpha}}$ as $\blambda\gg1$, the asymptotics of $u_{\blambda}(\rpal)$ is more complicated than previous three cases.

Firstly, by plugging \eqref{eq:3.14} into \eqref{eq:3.12}, one may obtain
\begin{equation}
\label{eq:3.22}
\begin{aligned}
u_{\blambda}(\rpal)=&\,2\log\sin\left(\frac{p\pi}{2\blambda^\alpha}+\frac{\pi}{\blambda}(1+o_{\blambda})\right)
\\=&\,2\log\left[\left(\frac{p\pi}{2\blambda^\alpha}+\frac{\pi}{\blambda}(1+o_{\blambda})\right)-\frac{1}{6}\left(\frac{p\pi}{2\blambda^\alpha}+\frac{\pi}{\blambda}(1+o_{\blambda})\right)^3(1+o_{\blambda})\right].
\end{aligned}
\end{equation}
Note also that $\frac{1}{\blambda^{3\alpha}}\ll\frac{1}{\blambda}$ as $3\alpha>1$, and $\frac{1}{\blambda^{3\alpha}}\gg\frac{1}{\blambda}$ as $0<3\alpha<1$.
We shall deal with \eqref{eq:3.22} under three situations $0<\alpha<1/3$, $\alpha=1/3$ and $\alpha>1/3$.
After simple calculations, we can establish the precise first three terms of $u_{\blambda}(\rpal)$ as follows:
\[
\begin{aligned}
u_{\blambda}(\rpal)&=
\begin{cases}\displaystyle
2\log\left(\frac{p\pi}{2\blambda^\alpha}+\frac{\pi}{\blambda}\left(1+o_{\blambda}\right)\right)&\mbox{if }\alpha\in\left(1/3,1\right),\\[3mm]\displaystyle
2\log\left(\frac{p\pi}{2\blambda^{\alpha}}+\frac{48\pi-p^3\pi^3}{48\blambda}\left(1+o_{\blambda}\right)\right)&\mbox{if }\alpha=1/3,\\[3mm]\displaystyle
2\log\left(\frac{p\pi}{2\blambda^\alpha}-\frac{p^3\pi^3}{48\blambda^{3\alpha}}\left(1+o_{\blambda}\right)\right)&\mbox{if }\alpha\in\left(0,1/3\right),
\end{cases}\\&=\log\frac{1}{\blambda^{2\alpha}}+\log\frac{p^2\pi^2}{4}(1+o_{\blambda}).\end{aligned}
\]
Thus, we obtain \eqref{eq:2.10}.

To prove \eqref{eq:2.11}, we shall deal with the asymptotics of \eqref{eq:3.13} with $0<\alpha<1$.
The argument is similar as the previous cases.
Indeed, by \eqref{eq:3.14} one may check that
\begin{equation}
\label{eq:3.24}
\begin{aligned}
\tan\left(\rpal\sqrt{\frac{J_{\blambda}}2}\right)=&\tan\left[\left(1-\frac{p}{\blambda^\alpha}\right)\left(\frac{\pi}2-\frac{\pi}\blambda(1+o_{\blambda})\right)\right]=\cot\left(\frac{p\pi}{2\blambda^\alpha}+\frac{\pi}\blambda(1+o_{\blambda})\right)
\\=&\frac{1}{\frac{p\pi}{2\blambda^\alpha}+\frac{\pi}{\blambda}(1+o_{\blambda})}=\frac{2\blambda^\alpha}{p\pi}+o_{\blambda}.
\end{aligned}
\end{equation}
Combining \eqref{eq:3.03}, \eqref{eq:3.13} and \eqref{eq:3.24}, we obtain
\[
u_{\blambda}'(\rpal)=\left(-\pi+\frac{2\pi}{\blambda}(1+o_{\blambda})\right)\left(\frac{2\blambda^\alpha}{p\pi}+o_{\blambda}\right)=-\frac{2}{p}\blambda^\alpha+o_{\blambda}.
\]
Therefore, we get \eqref{eq:2.11} and complete the proof of Theorem~\ref{thm:2.2}. 

\subsection{Proof of Corollary~\ref{coro:2.3}}
\label{subsec:3.2}

\eqref{eq:2.12} follows from the combination of \eqref{eq:3.01} and \eqref{eq:3.03} as follows:
\[
\begin{aligned}
\rho_{\blambda}(r)=&\left(\frac{\pi^2}{2}-\frac{2\pi^2}{\blambda}+\frac{6\pi^2}{\blambda^2}\left(1+o_{\blambda}\right)\right)\sec^2\left(\frac{\pi r}{2}-\frac{\pi r}{\blambda}+\frac{2\pi r}{\blambda^2}(1+o_{\blambda})\right)
\\=&\frac{\pi^2}2\sec^2\left(\frac{\pi}2r\right)-\frac{1}{\blambda}\left[2\pi^2\sec^2\left(\frac{\pi }{2}r\right)+\pi^3\sec^2\left(\frac{\pi }{2}r\right)\tan\left(\frac{\pi }{2}r\right)\right](1+o_{\blambda}),\end{aligned}
\]
This completes the proof of Corollary~\ref{coro:2.3}(a).

Now, we shall prove \eqref{eq:2.13}.
Putting $r=\rpal$ into the expression of $\rho_{\blambda}$ in \eqref{eq:3.01} and using \eqref{eq:3.03}, we can obtain
\begin{equation}
\label{eq:3.27}
\begin{aligned}
\rho_{\blambda}(\rpal)&=J_{\blambda}\sec^2\left(\rpal\sqrt{\frac{J_{\blambda}}2}\right)\\&=\left(\frac{\pi^2}{2}-\frac{2\pi^2}{\blambda}+\frac{6\pi^2}{\blambda^2}\left(1+o_{\blambda}\right)\right)\sec^2\left[\left(\frac{\pi}{2}-\frac{\pi}{\blambda}+\frac{2\pi}{\blambda^2}(1+o_{\blambda})\right)\left(1-\frac{p}{\blambda^{\alpha}}\right)\right].
\end{aligned}
\end{equation}
Note that the expansion of $\left(\frac{\pi}{2}-\frac{\pi}{\blambda}+\frac{2\pi}{\blambda^2}(1+o_{\blambda})\right)\left(1-\frac{p}{\blambda^{\alpha}}\right)$ depends variously on $\alpha\in\left(0,1\right)$, $\alpha=1$ and $\alpha\in\left(1,\infty\right)$.
Firstly, we deal with \eqref{eq:3.27} for the case of $0<\alpha<1$ and obtain
\[
\begin{aligned}
\rho_{\blambda}(\rpal)&=\left(\frac{\pi^2}{2}-\frac{2\pi^2}{\blambda}+\frac{6\pi^2}{\blambda^2}(1+o_{\blambda})\right)\sec^2\left(\frac{\pi}{2}-\frac{p\pi}{2\blambda^{\alpha}}+\frac{\pi}{\blambda}(1+o_{\blambda})\right)
\\&=\frac{2\blambda^{2\alpha}}{p^2}(1+o_{\blambda}),
\end{aligned}
\]
Here we have used the standard expansion of the cosecant function to obtain the last identity. Hence, we obtain \eqref{eq:2.13} for the case of $0<\alpha<1$.
By a similar argument, we can also prove \eqref{eq:2.13} for the two cases $\alpha=1$ and $\alpha>1$, and the proof of Corollary~\ref{coro:2.3}(b) is complete.

It remains to prove Corollary~\ref{coro:2.3}(c).
Note that $\displaystyle\lim_{\blambda\to\infty}\frac1{\int_0^1e^{-u_{\blambda}(s)}\mathrm{d}s}=0$ (by \eqref{eq:2.01}).
Along with \eqref{eq:1.05}, we arrive at\[\limsup_{\blambda\to\infty}\left(\frac{\rho_{\blambda}(r)h(r)}{\blambda}-\frac{u_{\blambda}'^2(r)h(r)}{2\blambda}\right)=\lim_{\blambda\to\infty}\frac{h(r)}{\int_0^1e^{-u_{\blambda}(s)}\mathrm{d}s}=0~~\mbox{uniformly~on~}\left[0,1\right]~\mbox{as}~\blambda\to\infty,\]
where $h$ is a continuous function on $\left[0,1\right]$.
This indicates that \eqref{eq:2.14} and \eqref{eq:2.15} are equivalent.
Hence, it suffices to claim \eqref{eq:2.14}, i.e.,
\begin{equation}
\label{eq:3.29}
\lim_{\blambda\to\infty}\int_0^1\frac{e^{-u_{\blambda}(r)}h(r)}{\int_0^1e^{-u_{\blambda}(s)}\mathrm{d}s}\mathrm{d}r=h(1).
\end{equation}
Let $\kappa\in(0,1)$ be fixed.
Then we observe that
\begin{equation}
\label{eq:3.30}
\begin{aligned}
\left|\int_0^1\frac{e^{-u_{\blambda}(r)}}{\int_0^1e^{-u_{\blambda}(s)}\mathrm{d}s}h(r)\mathrm{d}r-h(1)\right|&=\left|\left(\int_0^{1-\blambda^{-\kappa}}+\int_{1-\blambda^{-\kappa}}^1\right)\frac{e^{-u_{\blambda}(r)}}{\int_0^1e^{-u_{\blambda}(s)}\mathrm{d}s}(h(r)-h(1))\mathrm{d}r\right|
\\&\leq2\max_{r\in[0,1]}|h(r)|\left(\int_0^{1-\blambda^{-\kappa}}\frac{e^{-u_{\blambda}(r)}}{\int_0^1e^{-u_{\blambda}(s)}\mathrm{d}s}\mathrm{d}r\right)\\&~~~+\max_{r\in[1-\blambda^{-\kappa},1]}|h(r)-h(1)|\left(\int_{1-\blambda^{-\kappa}}^1\frac{e^{-u_{\blambda}(r)}}{\int_0^1e^{-u_{\blambda}(s)}\mathrm{d}s}\mathrm{d}r\right).
\end{aligned}
\end{equation}
Moreover, from \eqref{eq:1.05} and Theorem~\ref{thm:2.1}(d), a direct computation gives
\[\lim_{\blambda\to\infty}\int_{1-\blambda^{-\kappa}}^1\frac{e^{-u_{\blambda}(r)}}{\int_0^1e^{-u_{\blambda}(s)}\mathrm{d}s}\mathrm{d}r=-\lim_{\blambda\to\infty}\int_{1-\blambda^{-\kappa}}^1\frac{u_{\blambda}''(r)}{\blambda}\mathrm{d}r=\lim_{\blambda\to\infty}\frac{u_{\blambda}'(1-\blambda^{-\kappa})-u_{\blambda}'(1)}{\blambda}=1,
\]
which also implies
\begin{equation}
\label{eq:3.32}
\lim_{\blambda\to\infty}\int_0^{1-{\blambda}^{-\kappa}}\frac{e^{-u_{\blambda}(r)}}{\int_0^1e^{-u_{\blambda}(s)}\mathrm{d}s}\mathrm{d}r=0.
\end{equation}
Finally, we notice that the continuity of $h$ implies $\displaystyle\lim_{\blambda\to\infty}\max_{r\in\left[1-{\blambda}^{-\kappa},1\right]}|h(r)-h(1)|=0$.
As a consequence, \eqref{eq:3.29} immediately follows \eqref{eq:3.30}--\eqref{eq:3.32} and we prove Corollary~\ref{coro:2.3}(c).

Therefore, the proof of Corollary~\ref{coro:2.3} is completed.

\section{Proof of Theorem~\ref{thm:1.1}}
\label{sec4}

In order to deal with the convergence of the $v_{\bmu,\blambda}-u_{\blambda}$ with respect to $\bmu\blambda\to0^+$, throughout the whole section we shall set
\begin{equation}
\label{eq:4.01}
w_{\bmu,\blambda}:=v_{\bmu,\blambda}-u_{\blambda}.
\end{equation}
Subtracting \eqref{eq:1.05} from \eqref{eq:1.20} and using \eqref{eq:4.01} gives
\begin{equation}
\label{eq:4.02}
w_{\bmu,\blambda}''\left(r\right)=\frac{\bmu e^{v_{\bmu,\blambda}(r)}}{\int_0^1e^{v_{\bmu,\blambda}(s)}\mathrm{d}s}+\blambda\left(\frac{e^{-u_{\blambda}(r)}}{\int_0^1e^{-u_{\blambda}(s)}\mathrm{d}s}-\frac{e^{-v_{\bmu,\blambda}(r)}}{\int_0^1e^{-v_{\bmu,\blambda}(s)}\mathrm{d}s}\right)\end{equation}
and
\begin{equation}\label{eq:4.03}
w_{\bmu,\blambda}''(r)=e^{-v_{\bmu,\blambda}(r)}\left(\frac{\bmu e^{2v_{{\bmu},\blambda}(r)}}{\int_0^1e^{v_{\bmu,\blambda}(s)}\mathrm{d}s}-\frac{\blambda}{\int_0^1e^{-v_{\bmu,\blambda}(s)}\mathrm{d}s}+\frac{\blambda e^{w_{\bmu,\blambda}(r)}}{\int_0^1e^{-u_{\blambda}(s)}\mathrm{d}s}\right).
\end{equation}
We will sometimes use identity~\eqref{eq:4.02} or identity~\eqref{eq:4.03} to estimate $w_{\bmu,\blambda}$ and $w_{\bmu,\blambda}'$ for a sake of convenience.

Since we know that both $u_{\blambda}$ and $v_{\bmu,\blambda}$ are strictly decreasing on $[0,1]$ (by Lemma~\ref{lma:4.1}), the main difficulty of Theorem~\ref{thm:1.1} is to obtain the monotonicity of $w_{\bmu,\blambda}$, which will be presented in Lemma~\ref{lma:4.2}.
To complete the proof of Theorem~\ref{thm:1.1}(a), in Sections~\ref{subsec:4.2} and \ref{subsec:4.3},  we will prove
\begin{equation}
\label{eq:4.04}
\lim_{\mathop{\blambda\to\infty}\limits_{\mbox{\scriptsize $\bmu\blambda\to0$}}}\max_{\left[0,1\right]}\left|w_{\bmu,\blambda}\right|=0
\end{equation}
and
\begin{equation}
\label{eq:4.05}
\lim_{\mathop{\blambda\to\infty}\limits_{\mbox{\scriptsize $\bmu\blambda\to0$}}}\max_{\left[0,1\right]}\left|w_{\bmu,\blambda}'\right|=0,
\end{equation}
respectively.
When $\blambda>0$ is fixed and $\bmu\to0^+$, the proof of Theorem~\ref{thm:1.1}(b) is based on preliminary estimates in Section~\ref{subsec:4.1}--\ref{subsec:4.3}.
We will briefly state the proof in section~\ref{subsec:4.4}.

\subsection{Some basic properties}
\label{subsec:4.1}

Since $\blambda>\bmu>0$, one can follow the same argument as in  \cite[Lemma 2.1]{Lee2019} and \cite[Proposition 2.1]{Lee2011} to obtain $v_{\bmu,\blambda}''(r)=\frac{\bmu e^{v_{\bmu,\blambda}(r)}}{\int_0^1e^{v_{\bmu,\blambda}(s)}\text{d}s}-\frac{\blambda e^{-v_{\bmu,\blambda}(r)}}{\int_0^1e^{-v_{\bmu,\blambda}(s)}\text{d}s}<0$ for all $r\in(0,1]$.
Along with \eqref{eq:1.21}, it yields
\[v_{\bmu,\blambda}'(r)<v_{\bmu,\blambda}'(0)=0~\text{and}~v_{\bmu,\blambda}(r)<v_{\bmu,\blambda}(0)=0~\text{for all}~r\in(0,1].\]
As a consequence, we obtain the following property.
\begin{lma}
\label{lma:4.1}
For $\blambda>\bmu>0$, let $v_{\bmu,\blambda}\in\mathcal{C}^2([0,1])$ be the unique solution to \eqref{eq:1.20}--\eqref{eq:1.21}.
Then $v_{\bmu,\blambda}$ is strictly decreasing and strictly concave downward on $\left[0,1\right]$.\end{lma}

Lemma~\ref{lma:4.1} and \eqref{eq:1.08} present that $v_{\bmu,\blambda}$ and $u_{\blambda}$ are strictly decreasing on $[0,1]$, but do not provide further information for $w_{\bmu,\blambda}:=v_{\bmu,\blambda}-u_{\blambda}$. To prove \eqref{eq:4.04}, we need two crucial lemmas. Firstly, we obtain that $v_{\bmu,\blambda}>u_{\blambda}$ on $\left(0,1\right]$ and $w_{\bmu,\blambda}$ is {\bf monotonically increasing} as $\blambda>\bmu>0$, which is stated as follows.
\begin{lma}
\label{lma:4.2}
For $0<{\bmu}<\blambda$, $w_{{\bmu},\blambda}$ defined in \eqref{eq:4.01} is positive and monotonically increasing on $\left(0,1\right]$.
In particular, $w_{\bmu,\blambda}$ attains its maximum value at the boundary point $r=1$.
\end{lma}

\begin{proof}
Due to the continuity of $w_{\bmu,\blambda}$, there exists $r_1\in[0,1]$ such that $w_{\bmu,\blambda}$ attains its minimum value at $r=r_1$. Since $w_{\bmu,\blambda}'(1)={\bmu}>0$, we have $r_1\in\left[0,1\right)$.

We first show $r_1=0$. Suppose by contradiction that $r_1\in\left(0,1\right)$, which implies that $w_{\bmu,\blambda}(r_1)\leq0$, $w_{\bmu,\blambda}'(r_1)=0$ and $w_{\bmu,\blambda}''(r_1)\geq0$.
Since $v_{\bmu,\blambda}(r_1)<v_{\bmu,\blambda}(0)=0$ (by Lemma~\ref{lma:4.1}) and $w_{\bmu,\blambda}(r_1)\leq0$, it is easy to obtain
\begin{equation}
\label{eq:4.06}
\frac{\bmu}{\int_0^1e^{v_{\bmu,\blambda}(s)}\mathrm{d}s}+\frac{\blambda}{\int_0^1e^{-u_{\blambda}(s)}\mathrm{d}s}>\frac{\bmu e^{2v_{\bmu,\blambda}(r_1)}}{\int_0^1e^{v_{\bmu,\blambda}(s)}\mathrm{d}s}+\frac{\blambda e^{w_{\bmu,\blambda}(r_1)}}{\int_0^1e^{-u_{\blambda}(s)}\mathrm{d}s}.
\end{equation}
Along with \eqref{eq:4.03}, we find
\begin{equation}
\label{eq:4.07}
\begin{aligned}
w_{\bmu,\blambda}''(0)&=\frac{\bmu}{\int_0^1e^{v_{{\bmu},\blambda}(s)}\mathrm{d}s}-\frac{\blambda}{\int_0^1e^{u_{\blambda}(s)}\mathrm{d}s}+\frac{\blambda}{\int_0^1e^{-u_{\blambda}(s)}\mathrm{d}s}
\\&>\frac{\bmu e^{2v_{\bmu,\blambda}(r_1)}}{\int_0^1e^{v_{\bmu,\blambda}(s)}\mathrm{d}s}-\frac{\blambda}{\int_0^1e^{u_{\blambda}(s)}\mathrm{d}s}+\frac{\blambda e^{w_{\bmu,\blambda}(r_1)}}{\int_0^1e^{-u_{\blambda}(s)}\mathrm{d}s}=e^{v_{\bmu,\blambda}(r_1)}w_{\bmu,\blambda}''(r_1)>0.
\end{aligned}
\end{equation} 
Recall that $w_{\bmu,\blambda}'(0)=0$.
Hence, by \eqref{eq:4.07}, there exists $r_2\in(0,r_1)$ such that $w_{\bmu,\blambda}(r_2)=\max\limits_{\left[0,r_1\right]}w_{\bmu,\blambda}>0$.
In particular, 
\begin{equation}
\label{eq:4.08}
w_{\bmu,\blambda}''(r_2)\leq0.\end{equation}
Since $v_{\bmu,\blambda}(r_2)>v_{\bmu,\blambda}(r_1)$  and $w_{{\bmu},\blambda}(r_2)>0\geq w_{\bmu,\blambda}(r_1)$, we can repeat the same argument as in \eqref{eq:4.06} to obtain
\[
\frac{\bmu e^{2v_{\bmu,\blambda}(r_2)}}{\int_0^1e^{v_{{\bmu},\blambda}(s)}\mathrm{d}s}+\frac{\blambda e^{w_{\bmu,\blambda(r_2)}}}{\int_0^1e^{-u_{\blambda}(s)}\mathrm{d}s}>\frac{\bmu e^{2v_{\bmu,\blambda}(r_1)}}{\int_0^1e^{v_{{\bmu},\blambda}(s)}\mathrm{d}s}+\frac{\blambda e^{w_{\bmu,\blambda(r_1)}}}{\int_0^1e^{-u_{\blambda}(s)}\mathrm{d}s}.\]
Along with \eqref{eq:4.03}, we can follow the similar argument as in \eqref{eq:4.07} to get $w_{\bmu,\blambda}''(r_2)>0$, which contradicts \eqref{eq:4.08}. Thus, $r_1=0$ and $w_{\bmu,\blambda}\geq0$ on $\left[0,1\right]$.

Next, we want to show that $w_{\bmu,\blambda}$ attains its absolute maximum value at $r=1$.
Suppose by contradiction that there exists $r_3\in\left(0,1\right)$ such that $w_{\bmu,\blambda}$ attains its local maximum at $r=r_3$.
In particular, $w_{\bmu,\blambda}(r_3)>0$ and $w_{\bmu,\blambda}''(r_3)\leq0$, and there exists $\delta^*>0$ such that $w_{\bmu,\blambda}(r)\leq  w_{\bmu,\blambda}(r_3)$ and $w_{\bmu,\blambda}'(r)\leq0$ for $r\in\left(r_3,r_3+\delta^*\right)$.
On the other hand, since $w_{\bmu,\blambda}'(1)=\bmu>0$, there exists $r_4\in\left(r_3,1\right)$ such that $w_{\bmu,\blambda}$ attains its local minimum at $r=r_4$ with $w_{\bmu,\blambda}''(r_4)\geq0$.
Note that $w_{\bmu,\blambda}(r_4)\leq w_{\bmu,\blambda}(r_3)$ and $v_{\bmu,\blambda}(r_4)<v_{\bmu,\blambda}(r_3)$.
Thus, we can apply the similar argument as in \eqref{eq:4.06} and \eqref{eq:4.07} to get $w_{\bmu,\blambda}''(r_3)>0$, which contradicts the fact $w_{\bmu,\blambda}''(r_3)\leq0$.
As a consequence, $w_{\bmu,\blambda}$ attains its maximum value at $r=1$.
Furthermore, throughout the above argument, we also prove that $w_{\bmu,\blambda}$ has neither local maximum nor local minimum, and $w_{\bmu,\blambda}'$ preserves the same sign.
Consequently, $w_{\bmu,\blambda}$ is monotonically increasing since $w_{\bmu,\blambda}(0)<w_{\bmu,\blambda}(1)$.
Therefore, we complete the proof of Lemma~\ref{lma:4.2}.
\end{proof}

\subsection{Proof of \texorpdfstring{\eqref{eq:4.04}}{(4.4)}}
\label{subsec:4.2}

Thanks to Lemma~\ref{lma:4.2}, it suffices to show that $w_{\bmu,\blambda}(1)$ tends to zero as $\blambda\to\infty$ and $\bmu\blambda\to0$.
\begin{lma}
\label{lma:4.3}
If $\displaystyle\lim_{\blambda\to\infty}{\bmu}\blambda=0$, then
\begin{equation}
\label{eq:4.09}
\lim_{\blambda\to\infty}w_{\bmu,\blambda}(1)=0.\end{equation}
\end{lma}

\begin{proof}
Multiplying \eqref{eq:1.05} by $u_{\blambda}'$ and integrating the expression over $[0,r]$ gives
\begin{equation}
\label{eq:4.10}
\frac12u_{\blambda}'^2(r)=\frac{\blambda(e^{-u_{\blambda}(r)}-1)}{\int_0^1e^{-u_{\blambda}(s)}\mathrm{d}s}~~\mbox{for}~r\in\left[0,1\right].
\end{equation}
Applying the same argument to \eqref{eq:1.20}, we have
\begin{equation}
\label{eq:4.11}
\frac12v_{\bmu,\blambda}'^2(r)=\frac{{\bmu}(e^{v_{{\bmu},\blambda}(r)}-1)}{\int_0^1e^{v_{{\bmu},\blambda}(s)}\mathrm{d}s}+\frac{\blambda(e^{-v_{{\bmu},\blambda}(r)}-1)}{\int_0^1e^{-v_{{\bmu},\blambda}(s)}\mathrm{d}s}~~\mbox{for}~r\in\left[0,1\right].
\end{equation}
Putting $r=1$ into \eqref{eq:4.10} and \eqref{eq:4.11}, one arrives at
\begin{equation}
\label{eq:4.12}
\begin{aligned}
\left(1-\frac{{\bmu}}{\blambda}\right)^2&=\frac{v_{\bmu,\blambda}'^2(1)}{u_{\blambda}'^2(1)}\\&=\frac{\bmu}{\blambda}\frac{\int_0^1e^{-u_{\blambda}(s)}\mathrm{d}s}{\int_0^1e^{v_{\bmu,\blambda}(s)}\mathrm{d}s}\frac{e^{v_{\bmu,\blambda}(1)}-1}{e^{-u_{\blambda}(1)}-1}+\frac{\int_0^1e^{-u_{\blambda}(s)}\mathrm{d}s}{\int_0^1e^{-v_{\bmu,\blambda}(s)}\mathrm{d}s}\frac{e^{-v_{\bmu,\blambda}(1)}-1}{e^{-u_{\blambda}(1)}-1}
\\&<\frac{\int_0^1e^{-u_{\blambda}(s)}\mathrm{d}s}{\int_0^1e^{-v_{{\bmu},\blambda}(s)}\mathrm{d}s}\frac{e^{-w_{\bmu,\blambda}(1)}-e^{u_{\blambda}(1)}}{1-e^{u_{\blambda}(1)}}.
\end{aligned}
\end{equation}
Here we have used \eqref{eq:4.01} and the fact $v_{\bmu,\blambda}(1)<0$.
Since $-v_{\bmu,\blambda}\geq-w_{\bmu,\blambda}(1)-u_{\blambda}$ on $\left[0,1\right]$ (cf. Lemma~\ref{lma:4.2}), we have
\begin{equation}
\label{eq:4.13}
e^{-w_{\bmu,\blambda}(1)}\leq\frac{\int_0^1e^{-v_{\bmu,\blambda}(s)}\mathrm{d}s}{\int_0^1e^{-u_{\blambda}(s)}\mathrm{d}s}.
\end{equation}
Along with \eqref{eq:4.12}, we can obtain
\[\left(1-\frac{{\bmu}}{\blambda}\right)^2<\frac{1-e^{w_{{\bmu},\blambda}(1)+u_{\blambda}(1)}}{1-e^{u_{\blambda}(1)}},
\]
which implies
\begin{equation}
\label{eq:4.15}
1\leq e^{w_{\bmu,\blambda}(1)}<\left(1-\frac{\bmu}{\blambda}\right)^2+\left(2-\frac{\bmu}{\blambda}\right)\frac{\bmu}{\blambda}e^{-u_{\blambda}(1)}.\end{equation}
Note finally that \eqref{eq:2.05} implies the uniform boundness of $\frac{e^{-u_{\blambda}(1)}}{\blambda^2}$ with respect to $\blambda\gg1$.
Along with $\displaystyle\lim_{\blambda\to\infty}{\bmu}\blambda=0$, we get
\begin{equation}
\label{eq:4.16}
\lim_{\blambda\to\infty}\frac{{\bmu}}{\blambda}e^{-u_{\blambda}(1)}=0.
\end{equation}
Combining \eqref{eq:4.15} with \eqref{eq:4.16}, we deduce \eqref{eq:4.09} and complete the proof of Lemma~\ref{lma:4.3}.\end{proof}

Applying
Finally, by Lemmas~\ref{lma:4.2} and \ref{lma:4.3}, it is easy to see $\displaystyle\lim_{\mathop{\blambda\to\infty}\limits_{\mbox{\scriptsize $\bmu\blambda\to0$}}}\max_{r\in\left[0,1\right]}\left|w_{\bmu,\blambda}(r)\right|=\lim_{\mathop{\blambda\to\infty}\limits_{\mbox{\scriptsize $\bmu\blambda\to0$}}}w_{\bmu,\blambda}(1)=0$, which gives \eqref{eq:4.04}.

\subsection{Proof of \texorpdfstring{\eqref{eq:4.05}}{(4.5)}}
\label{subsec:4.3}

Note that $w_{\bmu,\blambda}'\geq0$, $w_{\bmu,\blambda}'(0)=0$ and $\displaystyle\lim_{\blambda\to\infty}w_{\bmu,\blambda}'(1)=\lim_{\blambda\to\infty}\bmu=0$.
For $\blambda\gg1\gg\bmu>0$, we may assume that $w_{\bmu,\blambda}'$ attains its maximum value at interior point $r_{\bmu,\blambda}^*$.
It suffices to claim
\begin{equation}
\label{eq:4.17}
\lim_{\mathop{\blambda\to\infty}\limits_{\mbox{\scriptsize $\bmu\blambda\to0$}}}w_{\bmu,\blambda}'(r_{\bmu,\blambda}^*)=0.
\end{equation}

{\bf Claim of \eqref{eq:4.17}.} Firstly, by $w_{\bmu,\blambda}''(r_{\bmu,\blambda}^*)=0$ and \eqref{eq:4.02}, we have
\begin{equation}
\label{eq:4.18}
0=w''_{\bmu,\blambda}(r_{\bmu,\blambda}^*)=\frac{\bmu e^{v_{\bmu,\blambda}(r_{\bmu,\blambda}^*)}}{\int_0^1e^{v_{\bmu,\blambda}(s)}\mathrm{d}s}-\frac{\blambda e^{-v_{\bmu,\blambda}(r_{\bmu,\blambda}^*)}}{\int_0^1e^{-v_{\bmu,\blambda}(s)}\mathrm{d}s}+\frac{\blambda e^{-u_{\blambda}(r_{\bmu,\blambda}^*)}}{\int_0^1e^{-u_{\blambda}(s)}\mathrm{d}s}.
\end{equation}
On the other hand, subtracting \eqref{eq:4.10} from \eqref{eq:4.11} gives
\begin{equation}
\label{eq:4.19}
\frac12\left(v_{\bmu,\blambda}'^2(r)-u_{\blambda}'^2(r)\right)=\frac{\bmu(e^{v_{\bmu,\blambda}(r)}-1)}{\int_0^1e^{v_{\bmu,\blambda}(s)}\mathrm{d}s}+\frac{\blambda(e^{-v_{\bmu,\blambda}(r)}-1)}{\int_0^1e^{-v_{\bmu,\blambda}(s)}\mathrm{d}s}-\frac{\blambda(e^{-u_{\blambda}(r)}-1)}{\int_0^1e^{-u_{\blambda}(s)}\mathrm{d}s}.
\end{equation}
Putting $r=r_{\bmu,\blambda}^*$ into \eqref{eq:4.19} and using \eqref{eq:4.18}, we observe that
\[
\begin{aligned}\frac12w_{\bmu,\blambda}'(r_{\bmu,\blambda}^*)&(v_{\bmu,\blambda}'(r_{\bmu,\blambda}^*)+u_{\blambda}'(r_{\bmu,\blambda}^*))
=\frac12\left(v_{\bmu,\blambda}'^2(r_{\bmu,\blambda}^*)-u_{\blambda}'^2(r_{\bmu,\blambda}^*)\right)
\\&=\frac{\bmu\left(e^{v_{\bmu,\blambda}(r_{\bmu,\blambda}^*)}-1\right)}{\int_0^1e^{v_{\bmu,\blambda}(s)}\mathrm{d}s}+\blambda\left(\frac{e^{-v_{\bmu,\blambda}(r_{\bmu,\blambda}^*)}}{\int_0^1e^{-v_{\bmu,\blambda}(s)}\mathrm{d}s}-\frac{e^{-u_{\blambda}(r_{\bmu,\blambda}^*)}}{\int_0^1e^{-u_{\blambda}(s)}\mathrm{d}s}\right)\\&\qquad-\blambda\left(\frac1{\int_0^1e^{-v_{\bmu,\blambda}(s)}\mathrm{d}s}-\frac1{\int_0^1e^{-u_{\blambda}(s)}\mathrm{d}s}\right)
\\&=\frac{\bmu(2e^{v_{\bmu,\blambda}(r_{\bmu,\blambda}^*)}-1)}{\int_0^1e^{v_{\bmu,\blambda}(s)}\mathrm{d}s}-\blambda\left(\frac1{\int_0^1e^{-v_{\bmu,\blambda}(s)}\mathrm{d}s}-\frac1{\int_0^1e^{-u_{\blambda}(s)}\mathrm{d}s}\right).
\end{aligned}
\]
Since $0\leq w_{\bmu,\blambda}'=v_{\bmu,\blambda}'-u_{\blambda}'\leq-\left(v_{\bmu,\blambda}'+u_{\blambda}'\right)$, we have
\begin{equation}
\label{eq:4.21}
\begin{aligned}
w_{\bmu,\blambda}'^2(r_{\bmu,\blambda}^*)&\leq-w_{\bmu,\blambda}'(r_{\bmu,\blambda}^*)(v_{\bmu,\blambda}'(r_{\bmu,\blambda}^*)+u_{\blambda}'(r_{\bmu,\blambda}^*))
\\&=-\frac{2\bmu(2e^{v_{\bmu,\blambda}(r_{\bmu,\blambda}^*)}-1)}{\int_0^1e^{v_{\bmu,\blambda}(s)}\mathrm{d}s}+2\blambda\left(\frac1{\int_0^1e^{-v_{\bmu,\blambda}(s)}\mathrm{d}s}-\frac1{\int_0^1e^{-u_{\blambda}(s)}\mathrm{d}s}\right).
\end{aligned}
\end{equation}
To deal with \eqref{eq:4.21}, we need the following lemma.

\begin{lma}
\label{lma:4.4} There hold

\begin{enumerate}
\item[(i)] $\displaystyle\lim_{\mathop{\blambda\to\infty}\limits_{\mbox{\scriptsize $\bmu\blambda\to0$}}}\blambda\left(\frac1{\int_0^1e^{-v_{\bmu,\blambda}(s)}\mathrm{d}s}-\frac1{\int_0^1e^{-u_{\blambda}(s)}\mathrm{d}s}\right)=0$.

\item[(ii)] $\displaystyle\lim_{\mathop{\blambda\to\infty}\limits_{\mbox{\scriptsize $\bmu\blambda\to0$}}}\frac{\bmu\left(2e^{v_{\bmu,\blambda}(r_{\bmu,\blambda}^*)}-1\right)}{\int_0^1e^{v_{\bmu,\blambda}(s)}\mathrm{d}s}=0$.
\end{enumerate}
\end{lma}

\begin{proof}
From Lemma~\ref{lma:4.2} and \eqref{eq:4.13}, we have\begin{equation}
\label{eq:4.22}\frac{\blambda}{\int_0^1e^{-u_{\bmu}(s)}\mathrm{d}s}<\frac{\blambda}{\int_0^1e^{-v_{\bmu,\blambda}(s)}\mathrm{d}s}<\frac{\blambda}{\int_0^1e^{-u_{\blambda}(s)}\mathrm{d}s}e^{w_{\bmu,\blambda}(1)}.\end{equation}
By \eqref{eq:2.01}, \eqref{eq:4.22} and Lemma~\ref{lma:4.3}, we get Lemma~\ref{lma:4.4}(i).

On the other hand, note that $\int_0^1e^{v_{\bmu,\blambda}(s)}\mathrm{d}s\int_0^1e^{-v_{\bmu,\blambda}(s)}\mathrm{d}s\geq1$ (by H\"older's inequality).
This along with \eqref{eq:2.01} and Lemma~\ref{lma:4.2} immediately implies
\[
0<\frac{\bmu}{\int_0^1e^{v_{\bmu,\blambda}(s)}\mathrm{d}s}\leq\bmu\int_0^1e^{-v_{\bmu,\blambda}(s)}\mathrm{d}s\leq \bmu\int_0^1e^{-u_{\blambda}(s)}\mathrm{d}s=\frac{2\bmu\blambda}{\pi^2}\left(1+o_{\blambda}\right)\stackrel{\bmu\blambda\to0}{-\!\!\!-\!\!\!-\!\!\!-\!\!\!-\!\!\!-\!\!\!\rightarrow}0.
\]
Since $v_{\bmu,\blambda}\leq0$ (by Lemma~\ref{lma:4.1}), we obtain
\[\left|\frac{\bmu\left(2e^{v_{\bmu,\blambda}(r_{\bmu,\blambda}^*)}-1\right)}{\int_0^1e^{v_{\bmu,\blambda}(s)}\mathrm{d}s}\right|\leq\frac{3\bmu}{\int_0^1e^{v_{\bmu,\blambda}(s)}\mathrm{d}s}\to0~~\mbox{as}~\blambda\to\infty.\]
This proves Lemma~\ref{lma:4.4}(ii). Therefore, the proof of Lemma~\ref{lma:4.4} is completed.
\end{proof}

Finally, by \eqref{eq:4.21} and Lemma~\ref{lma:4.4}, we arrive at \eqref{eq:4.17} and complete the proof of \eqref{eq:4.05}.

\subsection{Proof of Theorem \texorpdfstring{\ref{thm:1.1}(b)}{1.1(b)}}
\label{subsec:4.4}

In this section, we fix $\blambda>0$.
Note that Lemmas~\ref{lma:4.1}--\ref{lma:4.2} and the estimates \eqref{eq:4.10}--\eqref{eq:4.15} still hold for $\blambda>\bmu>0$.
Hence, as $\bmu\to0^+$, we have
\begin{equation}
\label{eq:4.24}
1\leq e^{w_{\bmu,\blambda}(1)}<\left(1-\frac{\bmu}{\blambda}\right)^2+\left(2-\frac{\bmu}{\blambda}\right)\frac{\bmu}{\blambda}e^{-u_{\blambda}(1)}\xrightarrow[\bmu\to0^+]{~\blambda>0~\text{fixed}~}1.
\end{equation}
Here we have used the fact that $u_{\blambda}(1)$ is independent of $\bmu$.
As a consequence, by \eqref{eq:4.24} and the monotonicity of $w_{\bmu,\blambda}$ (cf. Lemma~\ref{lma:4.2}), we arrive at
\begin{equation}
\label{eq:4.25}
\lim_{\bmu\to0^+}\max_{\left[0,1\right]}\left|w_{\bmu,\blambda}\right|=0.
\end{equation}

It remains to claim
\begin{equation}
\label{eq:4.26}
\lim_{\bmu\to0^+}\max_{\left[0,1\right]}\left|w_{\bmu,\blambda}'\right|=0.
\end{equation}
We assume that $w_{\bmu,\blambda}'$ attains its maximum value at $r=r_{\bmu}^*$.
If $r_{\bmu}\in\left(0,1\right)$, then one may check that \eqref{eq:4.18}--\eqref{eq:4.21} with $r_{\bmu,\blambda}^*=r_{\bmu}^*$ still hold for $\blambda>\bmu>0$.
Moreover, by \eqref{eq:4.22} and \eqref{eq:4.25}, we obtain
\begin{equation}
\label{eq:4.27}
\lim_{\bmu\to0^+}\frac{\blambda}{\int_0^1e^{-v_{\bmu,\blambda}(s)}\mathrm{d}s}=\frac{\blambda}{\int_0^1e^{-u_{\blambda}(s)}\mathrm{d}s}.
\end{equation}
This along with \eqref{eq:4.21} yields
\begin{equation}
\label{eq:4.28}
w_{\bmu,\blambda}'^2(r_{\bmu}^*)\leq\frac{\bmu}{\int_0^1e^{v_{\bmu,\blambda}(s)}\mathrm{d}s}+\blambda\left(\frac1{\int_0^1e^{-v_{\bmu,\blambda}(s)}\mathrm{d}s}-\frac1{\int_0^1e^{-u_{\blambda}(s)}\mathrm{d}s}\right)\xrightarrow[\bmu\to0^+]{~\blambda>0~\text{fixed}~}0.
\end{equation}
Therefore, we obtain \eqref{eq:4.26} and complete the proof of Theorem~\ref{thm:1.1}(b).

\section{Applications and discussion}\label{sec5}

In this section we provide an application to calculating capacitances for the doubler-layer capacitantors in single-ion electrolyte solutions. As was studied in \cite[Section 6]{Lee2019}, we define a quantity $\mathscr{C}^+(u_{\blambda};K)$ related to the capacitance in a physical region $K_{\blambda;\alpha}\subset[0,1]$ as
\begin{equation}
\label{eq:5.01}\mathscr{C}^+(u_{\blambda};K_{\blambda;\alpha})\coloneqq\frac{\left|\int_{K_{\blambda;\alpha}}\blambda^{-1}\rho_{\blambda}(r)\mathrm{d}r\right|}{\displaystyle\max_{x,y\in\overline{K_{\blambda;\alpha}}}\left|u_{\blambda}(x)-u_{\blambda}(y)\right|}.
\end{equation}
For a sake of simplicity, we shall set $K_{\blambda;\alpha}=[\rpal,1]$. We show that when $K_{\blambda;\alpha}$ attached to the boundary (the charge surface) has the thickness of the order $\blambda^
{-\alpha}$ and $\alpha\geq1$, $\mathscr{C}^+(u_{\blambda};K_{\blambda;\alpha})$ has a positive infimum as $\blambda$ tends to infinity. However, if the thickness of $K_{\blambda;\alpha}$ is far larger compared to the order $\blambda^{-1}$ as $\blambda\gg1$, then $\mathscr{C}^+(u_{\blambda};K_{\blambda;\alpha})$ tends to zero. Such results are based on the following refined asymptotics of $\mathscr{C}^+(u_{\blambda};[\rpal,1])$.

\begin{thm}\label{thm:5.1}Under the same hypothese as in Theorem~\ref{thm:2.1}, as $\blambda\gg1$ and $p>0$, the asymptotic expansions of $\mathscr{C}^+(u_{\blambda};[\rpal,1])$ are precisely depicted as follows:
\begin{enumerate}
\item[(a)] If $\alpha>1$, then
\[
\mathscr{C}^+\left(u_{\blambda};\left[\rpal,1\right]\right)=\frac12+\left(-\frac{p}{8\blambda^{\alpha-1}}\chi_1\left(\alpha\right)+\frac{\pi^2}{2\blambda^2}\chi_2\left(\alpha\right)\right)\left(1+o_{\blambda}\right),\]
where\[\chi_1\left(\alpha\right)=\begin{cases}0,&\text{if }\alpha>3,\\1,&\text{if }1<\alpha\leq3,\end{cases}\quad\mbox{and}\quad\chi_2\left(\alpha\right)=\begin{cases}1,&\text{if }\alpha\geq3,\\0,&\text{if }1<\alpha<3.\end{cases}\]
Note that $\chi_1\left(\alpha\right)=\chi_2\left(\alpha\right)=1$ as $\alpha=3$.
\item[(b)] If $\alpha=1$, then
\[
\mathscr{C}^+\left(u_{\blambda};\left[\rpal,1\right]\right)=\frac{p}{2\left(p+2\right)\log\left(1+\frac{p}{2}\right)}\left(1+\frac{H}{\blambda^2}(1+o_{\blambda})\right),
\]
where $H$ is defined by
\[H=\frac{\pi^2\left(p^2+6p+12\right)}{6\left(p+2\right)}+\frac{p^2\pi^2\left(p+6\right)}{24\left(p+2\right)\log\left(1+\frac{p}{2}\right)}.\]
\item[(c)] If $0<\alpha<1$, then
\[
\mathscr{C}^+\left(u_{\blambda};\left[\rpal,1\right]\right)=\frac{1}{\log\blambda^{2-2\alpha}}(1+o_{\blambda}).
\]
\end{enumerate}
\end{thm}\setlength{\parskip}{4pt}
Since the proof of Theorem~\ref{thm:5.1} requires a huge amount of elementary computations based on refined asymptotics of $u_{\blambda}$ and $u_{\blambda}'$ in Theorem~\ref{thm:2.2}, we omit the details.

Finally, we make brief summaries for Theorem~\ref{thm:5.1} as follows.
\begin{itemize}
\item 
For $\rpal$ defined in~\eqref{eq:1.14}, we have
\begin{equation}
\label{eq:5.05}
\mathscr{C}^+(u_{\blambda};[\rpal,1])=\begin{cases}
\displaystyle\frac{1}{2}(1+o_{\blambda})\quad&\text{for }\alpha>1,\\
\displaystyle\frac{p}{\displaystyle2(p+2)\log\left(1+\frac{p}{2}\right)}(1+o_{\blambda})\quad&\text{for }\alpha=1,\\
\displaystyle\frac{1}{\log\blambda^{2-2\alpha}}(1+o_{\blambda})\quad&\text{for }0<\alpha<1.\end{cases}
\end{equation}
\item For $r\in[0,1)$ independent of $\blambda$, $\displaystyle\mathscr{C}^+(u_{\blambda};[r,1])\sim\frac{1}{\log\blambda^2}$ tending to zero.
\end{itemize}
Note that in \eqref{eq:5.05}, $g\left(p\right):=\frac{p}{2(p+2)\log\left(1+\frac{p}{2}\right)}$ is strictly increasing to the variable $p>0$ and $\displaystyle\lim_{p\to\infty}g\left(p\right)=\textstyle\frac12$. Since the amount of electrical energy which the capacitor can store depends on its capacitance, \eqref{eq:5.05} confirms an important property of the ``double-layer capacitance" that the corresponding capacitance~\eqref{eq:5.01} of the electrostatic model~\eqref{eq:1.05}--\eqref{eq:1.06} stores much more energy in thinner region attached to the charged surface.

Before closing this section, we want to stress that the double-layer capacitance in binary electrolytes has been introduced in \cite[Theorem 6.1]{Lee2019}. Let us consider the same region $[\rpal,1]$ having the thickness $O(\blambda^{-\alpha})$ with $\alpha\geq1$ attached to the charged surface and explain why we are interested in calculating the corresponding capacitance in single-ion electrolytes. A reason is that for binary electrolytes, the maximum potential difference in $[\rpal,1]$ is too small to get the precise value of $\displaystyle\lim_{\blambda\to\infty}\mathscr{C}^+(u_{\blambda};[\rpal,1])$. However, for the case of single-ion electrolytes, we exactly obtain the precise value of $\displaystyle\lim_{\blambda\to\infty}\mathscr{C}^+(u_{\blambda};[\rpal,1])$ shown in \eqref{eq:5.05}. Such a result provides a practical application for calculating the double-layer capacitance in electrolytes~\cite{FJC}.

\section*{Acknowledgement} The authors are grateful to the referee for his/her carful reading and valuable suggestions which improve the exposition of the original manuscript.
This work was partially supported by the MOST grants~108-2115-M-007-006-MY2 (C.-C. Lee) and 106-2115-M-002-003-MY3 (T.-C. Lin) of Taiwan. The research of T.-C. Lin was also partially supported by the Center for Advanced Study in Theoretical Sciences (CASTS) and the National Center for Theoretical Sciences (NCTS) of Taiwan.

\section{Appendix: Proof of Lemma~\ref{lma:3.1}}

In this section, we state the proof of Lemma~\ref{lma:3.1}.

By \eqref{eq:3.01} and \eqref{eq:3.02}, one may check that
\[
I_{\blambda}
=\int_0^1e^{-u_{\blambda}(s)}\mathrm{d}s=\int_0^1\sec^2\left(\sqrt{\frac{J_{\blambda}}2}r\right)\mathrm{d}r
=\sqrt{\frac{2}{J_{\blambda}}}\tan\sqrt{\frac{J_{\blambda}}2},
\]
which implies
\begin{equation}
\label{eq:6.02}
\frac{\sqrt{2J_{\blambda}}}{\blambda}=\cot\sqrt{\frac{J_{\blambda}}2}=\tan\left(\frac{\pi}2-\sqrt{\frac{J_{\blambda}}2}\right).
\end{equation}

We shall now establish the precise first third order terms of the asymptotic expansion of $J_{\blambda}$ with respect to $\blambda\gg1$.
Firstly, by \eqref{eq:3.02} we have $0<J_{\blambda}<\pi^2/2$ for all $\blambda>0$.
Along with \eqref{eq:6.02}, it immediately yields
\[\lim_{\blambda\to\infty}J_{\blambda}=\frac{\pi^2}2.\]
This gives the precise leading order term of $J_{\blambda}$ with respect to $\blambda\gg1$.
Moreover, applying the approximation $\tan s=s+o(s)$ for $s=\frac{\pi}2-\sqrt{\frac{J_{\blambda}}2}\to0$ to the right-hand side of \eqref{eq:6.02}, one obtains
\begin{equation}
\label{eq:6.03}
\frac{\sqrt{2J_{\blambda}}}\blambda=\left(\frac{\pi}2-\sqrt{\frac{J_{\blambda}}2}\right)(1+o_{\blambda}).
\end{equation}

To deal with the second order term of $J_{\blambda}$ with respect to $\blambda\gg1$, let us set
\[a_{\blambda}=\left(J_{\blambda}-\frac{\pi^2}{2}\right){\blambda}.\]
Then we can express the asymptotic expansion of $\sqrt{\frac{J_{\blambda}}2}$ as 
\[
\sqrt{\frac{J_{\blambda}}2}=\sqrt{\frac{\pi^2}{4}+\frac{a_{\blambda}}{2\blambda}}=\frac{\pi}2+\frac{a_{\blambda}}{2\pi\blambda}(1+o_{\blambda}).
\]
Along with \eqref{eq:6.03} arrives at $\frac{\pi}2+\frac{a_{\blambda}}{2\pi\blambda}(1+o_{\blambda})=\frac{\blambda}{2}\left(\frac{\pi}2-\sqrt{\frac{J_{\blambda}}2}\right)(1+o_{\blambda})=-\frac{a_{\blambda}}{4\pi}(1+o_{\blambda})$, and consequently $a_{\blambda}=-2\pi^2+o_{\blambda}$, as $\blambda\gg1$, which stands the second order term of expansion of $J_{\blambda}$.
As a conclusion, 
\begin{equation}
\label{eq:6.05}
J_{\blambda}=\frac{\pi^2}2-\frac{2\pi^2}\blambda(1+o_{\blambda}).
\end{equation}
To further get the precise third order term of $J_{\blambda}$ with respect to $\blambda$, we consider the difference between $J_{\blambda}$ and its first two order terms shown in the right-hand side of \eqref{eq:6.05} and set
\[b_{\blambda}=\left(J_{\blambda}-\frac{\pi^2}2+\frac{2\pi^2}\blambda\right)\blambda^2.
\]
Then we have 
\begin{equation}
\label{eq:6.07}
\sqrt{\frac{J_{\blambda}}2}=\sqrt{\frac{\pi^2}{4}-\frac{\pi^2}{\blambda}+\frac{b_{\blambda}}{2\blambda^2}}=\frac{\pi}2-\frac{\pi}\blambda+\frac{1}{\blambda^2}\left(\frac{b_{\blambda}}{2\pi}-\pi\right)\left(1+o_{\blambda}\right).
\end{equation}
Rewriting \eqref{eq:6.03} as $\sqrt{\frac{J_{\blambda}}{2}}
=\frac{\blambda}{2}\left(\frac{\pi}2
-\sqrt{\frac{J_{\blambda}}2}\right)
(1+o_{\blambda})$ and putting \eqref{eq:6.07} into this expression, after a simple calculation we can get
\[
b_{\blambda}=6\pi^2+o_{\blambda},\,\,\text{as}\,\,\blambda\gg1.
\]

Similarly, to obtain the precise the fourth order term of $J_{\blambda}$ with respect to $\blambda$, we set
\begin{equation}
\label{eq:6.09}
c_{\blambda}=\blambda^3\left(J_{\blambda}-\frac{\pi^2}{2}+\frac{2\pi^2}{\blambda}-\frac{6\pi^2}{\blambda^2}\right).
\end{equation}
One may follow same argument to get
\begin{equation}\label{eq:6.10}
c_{\blambda}=-16\pi^2+\frac{2\pi^4}{3}+o_{\blambda},~\text{as }\blambda\gg1
\end{equation}
Therefore, \eqref{eq:3.03} directly follows from \eqref{eq:6.09} and  \eqref{eq:6.10}, and the proof of Lemma~\ref{lma:3.1} is complete.

\end{document}